\documentclass[12pt]{amsart}
\usepackage{epsfig}
\usepackage{graphics}
\usepackage{color}
\usepackage{dcpic, pictexwd}
\usepackage{tikz}
\usepackage[utf8]{inputenc}
\usepackage{amsmath}
\usepackage{amsfonts}
\usepackage{amssymb}
\usepackage{calligra}
\usepackage{calrsfs}
\usepackage[mathscr]{euscript}
\usepackage{pinlabel}

\newtheorem{theorem}{Theorem}

\newtheorem{lemma}[theorem]{Lemma}

\newtheorem{corollary}[theorem]{Corollary}

\theoremstyle{definition}

\theoremstyle{remark}

\newcommand{\blankbox}[2]{%
  \parbox{\columnwidth}{\centering
    \setlength{\fboxsep}{0pt}%
    \fbox{\raisebox{0pt}[#2]{\hspace{#1}}}%
  }%
}

 \def\doublespacing{\parskip 5 pt plus 1 pt \baselineskip 25 pt
      \lineskip 13 pt \normallineskip 13 pt}

 \newcommand{\abs}[1]{\lvert#1\rvert}
 \def\R{{\mathbb{R}}}
 \def\Z{{\mathbb{Z}}}
 \def\Q{{\mathbb{Q}}}
 \def\N{{\mathbb{N}}}
 \def\S{{\Sigma}}
 \def\C{{\mathbb{C}}}
 \def\mcg{{\rm MCG}}
\def\mod{{\rm Mod}}
\def\emod{{\rm Mod}^*}
 \def\GL{{\rm GL}}
 \def\SL{{\rm SL}}
 \def\Sp{{\rm Sp}}
 \def\Im{{\rm Im}}
 \def\PSp{{\rm PSp}}
 \def\G{{\mathcal{G}}}
\def\S{{\mathcal{S}}}

\begin{document}

\newenvironment{prooff}{\medskip \par \noindent {\it Proof}\ }{\hfill
$\square$ \medskip \par}
    \def\sqr#1#2{{\vcenter{\hrule height.#2pt
        \hbox{\vrule width.#2pt height#1pt \kern#1pt
            \vrule width.#2pt}\hrule height.#2pt}}}
    \def\square{\mathchoice\sqr67\sqr67\sqr{2.1}6\sqr{1.5}6}
\def\pf#1{\medskip \par \noindent {\it #1.}\ }
\def\endpf{\hfill $\square$ \medskip \par}
\def\demo#1{\medskip \par \noindent {\it #1.}\ }
\def\enddemo{\medskip \par}
\def\qed{~\hfill$\square$}

 \title[Generating Mapping Class Group By Two Torsion Elements]
{Generating Mapping Class Group By Two Torsion Elements}

\author[O\u{g}uz Y{\i}ld{\i}z]{O\u{g}uz Y{\i}ld{\i}z}
\date{\today}

 \address{Department of Mathematics, Middle East Technical University, 06800
  Ankara, Turkey
  }
 \email{oguzyildiz16@gmail.com}
\begin{abstract}
We prove that the mapping class group of a closed connected orientable surface of genus $g$
is generated by two elements of order $g$ for $g\geq 6$. Moreover, for $g\geq 7$ we found a generating set of two elements, of order $g$ and $g'$ which is the least divisor of $g$ such that $g'>2$.
Furthermore, we proved that it is generated by two elements of order $g/\gcd (g,k)$ for $g\geq 3k^2+4k+1$ and any positive integer $k$.
\end{abstract} 
 \maketitle

\section{Introduction}
The mapping class group $\mod(\Sigma_g)$ of a compact connected orientable 
surface $\Sigma_g$ without boundary is the group of  
orientation--preserving diffeomorphisms of $\Sigma_g\to \Sigma_g$ up to isotopy. This study aims to explore
generation of $\mod(\Sigma_g)$ by two torsion elements of small orders. 

Korkmaz is the first person to find a generating set consisting of two torsions, elements of finite order, for the mapping class group $\mod(\Sigma_g)$. He~\cite{Korkmaz} proved that the group $\mod(\Sigma_g)$ 
is generated by two elements of order $4g+2$. Korkmaz asked in~\cite{Korkmaz2012} the following problem: 
What are the other numbers $k$ (less than $4g+2$) such that $\mod(\Sigma_g)$ is generated by two elements of order $k$. What is the smallest such $k$?

Dan Margalit~\cite{Margalit} asked a quite similar question, too:
For which $k$ can $\mod(\Sigma_g)$ be generated by two elements of order $k$?

We proved on Theorem ~\ref{thm:2} that the mapping class group $\mod(\Sigma_g)$ is generated by 
two elements of order $g$ if $g \geq 6$. So, we showed that $k=g$ satisfies the desired result for $g \geq 6$.
Moreover, we found generating sets of two elements of smaller orders as shown in Theorem ~\ref{thm:1} and Theorem ~\ref{thm:3}.

\begin{theorem} \label{thm:2}
The mapping class group $\mod(\Sigma_g)$ is generated by 
two elements of order g for $g \geq 6$. 
\end{theorem}

\begin{theorem} \label{thm:1}
For $g\geq 7$ the mapping class group $\mod(\Sigma_g)$ is generated by 
two elements of order $g$ and order $g'$ where $g'$ is the least divisor of $g$ such that $g'>2$. 
\end{theorem}

\begin{theorem} \label{thm:3}
For $g\geq 3k^2+4k+1$ and any positive integer $k$, the mapping class group $\mod(\Sigma_g)$ is generated by 
two elements of order $g/$$\gcd(g,k)$.
\end{theorem}

Since there is a surjective 
homomorphism from $\mod(\Sigma_g)$ onto the symplectic group 
$\Sp (2g,\Z)$, 
we have the following immediate result: 

\begin{corollary} 
The symplectic group $\Sp(2g,\Z)$ is generated 
by two elements of order g for $g\geq 6$.
\end{corollary}

Firstly, Dehn~\cite{Dehn}
showed that $\mod(\Sigma_g)$ is generated by $2g(g-1)$ many Dehn twists. Lickorish~\cite{Lickorish} 
decreased this number to $3g-1$.
Humphries~\cite{Humphries} introduced a generating set of $2g+1$ many Dehn twists and proved that this is the least such number.
Lu~\cite{Lu} found a generating set of three elements and two of the generators are of finite order. 
  
Maclachlan~\cite{Maclachlan} proved that $\mod(\Sigma_g)$ can be generated by only using torsions. McCarthy and Papadopoulos~\cite{McCarthyPapadopoulos} showed
$\mod(\Sigma_g)$ can be generated by involutions. Stukow~\cite{Stukow2}
proved index five subgroup of $\mod(\Sigma_2)$ is generated by using involutions.
Luo~\cite{Luo} found an upper bound that is necessary to generate $\mod(\Sigma_g)$ by involutions.
After that, Brendle and Farb~\cite{BrendleFarb} showed that six involutions are enough. 
Kassabov~\cite{Kassabov} decreased this to $4$ for $g\geq 7$ . 
Recently, Korkmaz~\cite{Korkmaz3inv} introduced a generating set consisting of three involutions for $g\geq 8$ and of four involutions for $g\geq 3$.
Finally, we ~\cite{Yildiz} showed that $\mod(\Sigma_g)$ is generated by three involutions if $g\geq 6$.
Since the mapping class group cannot be generated by two involutions for homological reasons, three was the least possible such number. Now, we are interested in
generating $\mod(\Sigma_g)$ by two torsion elements of small orders. 

Wajnryb~\cite{Wajnryb1983} found a presentation for the mapping class group of an orientable surface.
Wajnryb~\cite{Wajnryb1996} also proved that
$\mod(\Sigma_g)$ can be generated by two elements; one is of order $4g+2$ and the other
is a product of opposite Dehn twists. 
After that, Korkmaz~\cite{Korkmaz} showed that $\mod(\Sigma_g)$ 
is generated by an element of order $4g+2$ and a Dehn twist, improving Wajnryb's result.
He also proved that $\mod(\Sigma_g)$ can be generated by two torsion elements of order $4g+2$. 
Recently, Baykur and Korkmaz~\cite{Baykur} proved that the mapping class group can be generated by two commutators.

Moreover, Monden~\cite{Monden} proved that $\mod(\Sigma_g)$ can be generated by three elements of order $3$ and by four elements of order $4$. 
Lanier~\cite{Lanier} showed $\mod(\Sigma_g)$ is generated by three elements of any order $k\geq 6$ if $g\geq (k-1)^2+1$.
See ~\cite{Du2} for generating set of torsion elements for the extended mapping class group.
%

\bigskip

\section{An overview of mapping class groups}

Throughout the paper only closed connected orientable surfaces, $\Sigma_g$,  
where all genera are distributed as in Figure ~\ref{fig1} are considered. Note that the rotation $R$ by $2\pi$$/g$ degrees about $z$-axis is a well-defined
self-diffeomorphism of $\Sigma_g$. The mapping class group $\mod(\Sigma_g)$ of a closed connected orientable 
surface $\Sigma_g$ is the group of orientation--preserving diffeomorphisms of $\Sigma_g\to \Sigma_g$ up to isotopy.
Diffeomorphisms and curves are classified up to isotopy. 
We refer to~\cite{FarbMargalit} for all other information on the mapping class groups.
We only deal with simple types of simple closed curves 
$a_i$'s, $b_i$'s and $c_i$'s as shown on Figure ~\ref{fig1} where $1 \leq i \leq g$.
In order to align with the notation in ~\cite{Yildiz}, 
we show simple closed curves by lowercase letters $a_i$, $b_i$, $c_i$
and corresponding positive Dehn twists by uppercase letters $A_i$, $B_i$, $C_i$ or the usual notation $t_{a_i}$,$t_{b_i}$,$t_{c_i}$, respectively.
All indices are to be considered as modulo $g$.
For the composition of diffeomorphisms, 
$f_1f_2$ means that $f_2$ is first and then $f_1$ comes second as usual.

Commutativity, Braid Relation and the following basic facts on the mapping class group are used along the paper for many times:
For any simple closed curves $c_1$ and $c_2$ on $\Sigma_g$ and diffeomorphism $f:\Sigma_g \to \Sigma_g$, 
$ft_{c_1}f^{-1}=t_{f(c_1)}$; $c_1$ is isotopic to $c_2$ if and only if $t_{c_1}=t_{c_2}$ in $\mod(\Sigma_g)$; and if $c_1$ and $c_2$ are disjoint, then $t_{c_1}(c_2)=c_2$.

After Dehn and Lickorish proved the mapping class group $\mod(\Sigma_g)$ is generated by $3g-1$ Dehn twists about nonseparating
simple closed curves,  Humphries~\cite{Humphries} stated the following theorem.
 
\begin{theorem} {\rm\bf(Dehn-Lickorish-Humphries)}\label{thm:Humphries}
The mapping class group $\mod(\Sigma_g)$ is generated by the set
$\{ A_1,A_2,B_1,B_2,\ldots ,B_g,C_1,C_2,\ldots ,C_{g-1}\}.$
\end{theorem}

Let $R$ denote the rotation by $2\pi/g$ about the $x$--axis represented in Figure~\ref{fig1}. 
Then, \mbox{$R(a_k)=a_{k+1}$,} $R(b_k)=b_{k+1}$ and
$R(c_k)=c_{k+1}$. 
Korkmaz~\cite{Korkmaz3inv} showed on the following handy theorem deducing from Theorem~\ref{thm:Humphries} that the mapping class group is generated the four elements. First element is the rotation element $R$
and others are products of one positive and one negative Dehn twists.
\begin{theorem} \label{thm:thmKorkmaz} 
If $g\geq 3$, then the mapping class group $\mod(\Sigma_g)$ is generated by the four elements
\(
R, A_1A_2^{-1},B_1B_2^{-1}, C_1C_2^{-1}. \)
\end{theorem}

The next result is easily deduced from Theorem ~\ref{thm:thmKorkmaz}.

\begin{corollary} \label{cor:thmKorkmaz} 
If $g\geq 3$, then the mapping class group $\mod(\Sigma_g)$ is generated by the four elements
\(
R, A_1B_1^{-1}, B_1C_1^{-1}, C_1B_2^{-1}. \)
\end{corollary}

\begin{proof}

Let $H$ be the subgroup of $\mod(\Sigma_g)$ generated by the set\\
$\{ R, A_1B_1^{-1}, B_1C_1^{-1}, C_1B_2^{-1}.\}$.

Then, $B_2A_2^{-1} = R(B_1A_1^{-1})R^{-1}$$\in$$H$ and 
$B_2C_2^{-1} = R(B_1C_1^{-1})R^{-1}\in$$H$.
$B_1B_2^{-1}=(B_1C_1^{-1})(C_1B_2^{-1})$$\in$$H$,
$C_1C_2^{-1}=(C_1B_2^{-1})(B_2C_2^{-1})$$\in$$H$ and
$A_1A_2^{-1}=(A_1B_1^{-1})(B_1B_2^{-1})(B_2A_2^{-1})$$\in$$H$.

It follows from  Theorem~\ref{thm:thmKorkmaz} that
$H=\mod(\Sigma_g)$, completing the proof of the corollary.
\end{proof}

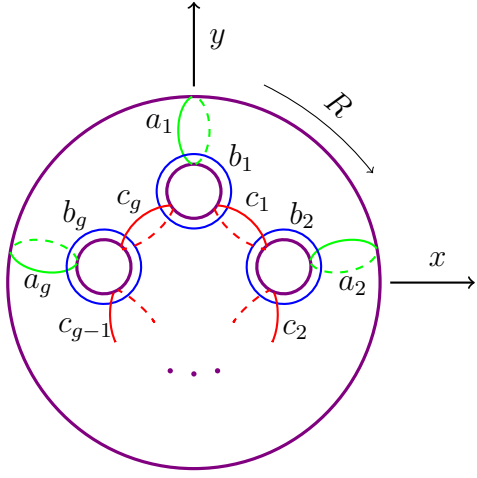
\begin{figure}
\begin{tikzpicture}[scale=0.45]
\begin{scope} [xshift=7cm]
 \draw[very thick, violet] (0,0) circle [radius=5.5cm];
 \draw[very thick, violet] (0,2.7) circle [radius=0.8cm]; 
 \draw[very thick, violet, rotate=80] (0,2.7) circle [radius=0.8cm]; 
 \draw[very thick, violet, rotate=-80] (0,2.7) circle [radius=0.8cm]; 
 \draw[very thick, violet, fill ] (0,-2.7) circle [radius=0.03cm]; 
 \draw[very thick, violet, fill, rotate=15] (0,-2.7) circle [radius=0.03cm]; 
\draw[very thick, violet, fill, rotate=-15] (0,-2.7) circle [radius=0.03cm]; 
\draw[thick, green,  rounded corners=10pt] (-0.05, 3.5) ..controls (-0.6,3.8) and (-0.6,5.2).. (-0.05,5.5) ;
\draw[thick, green, dashed, rounded corners=10pt] (0.05,3.5)..controls (0.6,3.8) and (0.6,5.2).. (0.05,5.5) ;
\node at (-1,4.7) {$a_1$};
\draw[thick, green, rotate=80, rounded corners=10pt] 
	(-0.05, 3.5) ..controls (-0.6,3.8) and (-0.6,5.2).. (-0.05,5.5) ;
\draw[thick, green, dashed, rotate=80, rounded corners=10pt] 
	(0.05,3.5)..controls (0.6,3.8) and (0.6,5.2).. (0.05,5.5) ;
\node at (-4.6,-0.1) {$a_g$};
\draw[thick, green, rotate=-80, rounded corners=10pt] 
	(-0.05, 3.5) ..controls (-0.6,3.8) and (-0.6,5.2).. (-0.05,5.5) ;
\draw[thick, green, dashed, rotate=-80, rounded corners=10pt] 
	(0.05,3.5)..controls (0.6,3.8) and (0.6,5.2).. (0.05,5.5) ;
\node at (4.7,-0.1) {$a_2$};
 \draw[thick, blue] (0,2.7) circle [radius=1.1cm]; 
 \draw[thick, rotate=80, blue] (0,2.7) circle [radius=1.1cm]; 
 \draw[thick, blue, rotate=-80] (0,2.7) circle [radius=1.1cm]; 
\node at (-3.5,2) {$b_g$};
\node at (3.2,2) {$b_2$};
\node at (1.4,3.7) {$b_1$};
\draw[thick, red, rounded corners=15pt] (-2.15, 1.05)-- (-1.856,2.233) -- (-0.7,2.28) ;
\draw[thick, red, dashed, rounded corners=15pt] (-2.03, 1)-- (-0.96,1.16) -- (-0.6,2.18) ;
\node at (-1.9,2.4) {$c_g$};
\draw[thick, red, rounded corners=15pt] (2.15, 1.05)-- (1.856,2.233) -- (0.7,2.28) ;
\draw[thick, red, dashed, rounded corners=15pt] (2.03, 1)-- (0.96,1.16) -- (0.6,2.18) ;
\node at (1.9,2.4) {$c_1$};
\draw[thick, red, rotate=80, rounded corners=10pt] (-2.1, 2)-- (-1.5,2.433) -- (-0.7,2.28) ;
\draw[thick, red, dashed, rotate=80, rounded corners=8pt] (-1.3, 1)-- (-0.9,1.36) -- (-0.6,2.18) ;
\node at (-3.2,-1.4) {$c_{g-1}$};
\draw[thick, red, rotate=-80, rounded corners=10pt] (2.1, 2)-- (1.5,2.433) -- (0.7,2.28) ;
\draw[thick, red, dashed, rotate=-80, rounded corners=8pt] (1.3, 1)-- (0.9,1.36) -- (0.6,2.18) ;
\node at (3,-1.4) {$c_2$};
\draw[->, rounded corners=20pt, rotate=-19.5] (0, 6.2)..controls (1.395, 6.039) and (2.716, 5.574).. (3.906,4.817);
\node[rotate=-40] at (4.28,5.284) {$R$};
\end{scope}
\node at (0.7+7, 6.7+0.5) {$y$};
\node at (6.7+7+0.5, 0.7) {$x$};
\draw[->, thick] (0+7,5.8)--(0+7, 8.3);
\draw[->, thick] (5.8+7,0)--(8.3+7,0);

\end{tikzpicture}
       \caption{The curves $a_i,b_i,c_i$ and the rotation $R$
          on the surface $\Sigma_g$.}
  \label{fig1}
\end{figure}

\bigskip


\section{Twelve new generating sets for $\mod(\Sigma_g)$.}
In this section, we obtain twelve new generating sets to 
generate the mapping class group by two elements of small orders.
We follow the idea of Korkmaz in~\cite{Korkmaz3inv} to create our generating sets in the next corollaries to Theorem~\ref{thm:thmKorkmaz}.
Our idea is to use the rotation element $R$, a torsion element of order $g$ in the group $\mod(\Sigma_g)$, as the first element and products of Dehn twists as the second element. 

We use the first four corollaries to create generating sets of elements of order $g$.
In order to follow Theorem~\ref{thm:thmKorkmaz}, we need the rotation element $R$, which is of order $g$, in our generating set. So, we cannot decrease order $g$ by following this method.
But indeed, we can reduce the order of the second element. Applying Lemma~\ref{lem:order}, we see that most suitable candidates are the divisors of $g$. Unfortunately, it is not possible 
to find an order $2$ element, an involution, as the second element in our generating sets by following this method. Instead, we found the least divisor of $g$ greater than $2$ as the minimal order for the second element. 
We use the next six corollaries after the first four ones to create generating sets of elements of order $g$ and $g'$ where $g'$ is the least divisor of $g$ such that $g'>2$.

\begin{corollary} \label{cor:gen1} 
If $g=6$, then the mapping class group  $\mod(\Sigma_g)$ is generated by 
the two elements
$R$ and $C_1B_4A_6A_1^{-1}B_5^{-1}C_2^{-1}$.
\end{corollary}

\begin{proof}
Let $F_1=C_1B_4A_6A_1^{-1}B_5^{-1}C_2^{-1}$. Let us denote by
$H$ the subgroup of $\mod(\Sigma_6)$ generated by the set
$\{ R, F_1\}$.

If $H$ contains the elements $A_1A_2^{-1}$, $B_1B_2^{-1}$ and $C_1C_2^{-1}$ then we are done by Theorem~\ref{thm:thmKorkmaz}. 

Let
\begin{eqnarray*}
F_2
&=& RF_1R^{-1} \\
&=& R(C_1B_4A_6A_1^{-1}B_5^{-1}C_2^{-1})R^{-1}\\
&=& RC_1R^{-1}RB_4R^{-1}RA_6R^{-1}RA_1^{-1}R^{-1}RB_5^{-1}R^{-1}RC_2^{-1}R^{-1}\\
&=& Rt_{c_1}R^{-1}Rt_{b_4}R^{-1}Rt_{a_6}R^{-1}Rt_{a_1}^{-1}R^{-1}Rt_{b_5}^{-1}R^{-1}Rt_{c_2}^{-1}R^{-1}\\
&=& t_{R(c_1)}t_{R(b_4)}t_{R(a_6)}t_{R(a_1)}^{-1}t_{R(b_5)}^{-1}t_{R(c_2)}^{-1}\\
&=& t_{c_2}t_{b_5}t_{a_1}t_{a_2}^{-1}t_{b_6}^{-1}t_{c_3}^{-1}\\
&=& C_2B_5A_1A_2^{-1}B_6^{-1}C_3^{-1}
\end{eqnarray*}
and \[
F_3 = F_2^{-1} = C_3B_6A_2A_1^{-1}B_5^{-1}C_2^{-1}.
\]

\begin{figure}
\begin{tikzpicture}[scale=0.7]
\begin{scope} [xshift=0cm, yshift=0cm]
 \draw[very thick, violet] (0,0) circle [radius=2.5cm];
 \draw[very thick, violet] (0,1.6) circle [radius=0.2cm]; 
 \draw[very thick, violet, rotate=60] (0,1.6) circle [radius=0.2cm]; 
 \draw[very thick, violet, rotate=120] (0,1.6) circle [radius=0.2cm];  
 \draw[very thick, violet, rotate=180] (0,1.6) circle [radius=0.2cm];  
 \draw[very thick, violet, rotate=-120] (0,1.6) circle [radius=0.2cm]; 
 \draw[very thick, violet, rotate=-60] (0,1.6) circle [radius=0.2cm]; 
 \draw[thick, blue, rotate=-60, rounded corners=8pt] (-1.2,0.94)--(-0.8, 1.04)--(-0.18,1.5);  
 \draw[thick, blue, dashed,  rotate=-60, rounded corners=8pt] (-1.2,0.98)--(-0.92, 1.28)--(-0.2,1.54);  
 \draw[thick, red, rotate=-120, rounded corners=8pt] (-1.2,0.94)--(-0.8, 1.04)--(-0.18,1.5);  
 \draw[thick, red, dashed,  rotate=-120, rounded corners=8pt] (-1.2,0.98)--(-0.92, 1.28)--(-0.2,1.54);  
 \draw[thick, blue, rotate=180] (0,1.6) circle [radius=0.3cm]; 
\draw[thick, red, rotate=120] (0,1.6) circle [radius=0.3cm]; 
 \draw[thick, blue, rotate=60, rounded corners=8pt] (-0.02,1.8)--(-0.1, 2.15)--(-0.02,2.5);  
 \draw[thick, blue, dashed, rotate=60, rounded corners=8pt]  (0.02,1.8)--(0.1, 2.15)--(0.02,2.5); 
 \draw[thick, dashed, red, rotate=0, rounded corners=8pt] (-0.02,1.8)--(-0.1, 2.15)--(-0.02,2.5);  
 \draw[thick, red, rotate=0, rounded corners=8pt]  (0.02,1.8)--(0.1, 2.15)--(0.02,2.5); 
  \node[scale=0.6, red] at (-0.3,2.1) {$-$};
  \node[scale=0.6, blue] at (0.5,1.0) {$+$};
  \node[scale=0.6, red] at (1.1,0.0) {$-$};
  \node[scale=0.6, blue] at (-2.0,0.9) {$+$};
  \node[scale=0.6, blue] at (0.0,-1.1) {$+$};
  \node[scale=0.6, red] at (-0.9,-0.7) {$-$};
\end{scope}

\begin{scope} [xshift=5.8cm, yshift=0cm, rotate=-60]
 \draw[very thick, violet] (0,0) circle [radius=2.5cm];
 \draw[very thick, violet] (0,1.6) circle [radius=0.2cm]; 
 \draw[very thick, violet, rotate=60] (0,1.6) circle [radius=0.2cm]; 
 \draw[very thick, violet, rotate=120] (0,1.6) circle [radius=0.2cm];  
 \draw[very thick, violet, rotate=180] (0,1.6) circle [radius=0.2cm];  
 \draw[very thick, violet, rotate=-120] (0,1.6) circle [radius=0.2cm]; 
 \draw[very thick, violet, rotate=-60] (0,1.6) circle [radius=0.2cm]; 
 \draw[thick, blue, rotate=-60, rounded corners=8pt] (-1.2,0.94)--(-0.8, 1.04)--(-0.18,1.5);  
 \draw[thick, blue, dashed,  rotate=-60, rounded corners=8pt] (-1.2,0.98)--(-0.92, 1.28)--(-0.2,1.54);  
 \draw[thick, red, rotate=-120, rounded corners=8pt] (-1.2,0.94)--(-0.8, 1.04)--(-0.18,1.5);  
 \draw[thick, red, dashed,  rotate=-120, rounded corners=8pt] (-1.2,0.98)--(-0.92, 1.28)--(-0.2,1.54);  
 \draw[thick, blue, rotate=180] (0,1.6) circle [radius=0.3cm]; 
\draw[thick, red, rotate=120] (0,1.6) circle [radius=0.3cm]; 
 \draw[thick, blue, rotate=60, rounded corners=8pt] (-0.02,1.8)--(-0.1, 2.15)--(-0.02,2.5);  
 \draw[thick, blue, dashed, rotate=60, rounded corners=8pt]  (0.02,1.8)--(0.1, 2.15)--(0.02,2.5); 
 \draw[thick, dashed, red, rotate=0, rounded corners=8pt] (-0.02,1.8)--(-0.1, 2.15)--(-0.02,2.5);  
 \draw[thick, red, rotate=0, rounded corners=8pt]  (0.02,1.8)--(0.1, 2.15)--(0.02,2.5); 
  \node[scale=0.6, red] at (-0.3,2.1) {$-$};
  \node[scale=0.6, blue] at (0.5,1.0) {$+$};
  \node[scale=0.6, red] at (1.1,0.0) {$-$};
  \node[scale=0.6, blue] at (-2.0,0.9) {$+$};
  \node[scale=0.6, blue] at (0.0,-1.1) {$+$};
  \node[scale=0.6, red] at (-0.9,-0.7) {$-$};
\end{scope}

\begin{scope} [xshift=11.6cm, yshift=0cm, rotate=-60]
 \draw[very thick, violet] (0,0) circle [radius=2.5cm];
 \draw[very thick, violet] (0,1.6) circle [radius=0.2cm]; 
 \draw[very thick, violet, rotate=60] (0,1.6) circle [radius=0.2cm]; 
 \draw[very thick, violet, rotate=120] (0,1.6) circle [radius=0.2cm];  
 \draw[very thick, violet, rotate=180] (0,1.6) circle [radius=0.2cm];  
 \draw[very thick, violet, rotate=-120] (0,1.6) circle [radius=0.2cm]; 
 \draw[very thick, violet, rotate=-60] (0,1.6) circle [radius=0.2cm]; 
 \draw[thick, red, rotate=-60, rounded corners=8pt] (-1.2,0.94)--(-0.8, 1.04)--(-0.18,1.5);  
 \draw[thick, red, dashed,  rotate=-60, rounded corners=8pt] (-1.2,0.98)--(-0.92, 1.28)--(-0.2,1.54);  
 \draw[thick, blue, rotate=-120, rounded corners=8pt] (-1.2,0.94)--(-0.8, 1.04)--(-0.18,1.5);  
 \draw[thick, blue, dashed,  rotate=-120, rounded corners=8pt] (-1.2,0.98)--(-0.92, 1.28)--(-0.2,1.54);  
 \draw[thick, red, rotate=180] (0,1.6) circle [radius=0.3cm]; 
\draw[thick, blue, rotate=120] (0,1.6) circle [radius=0.3cm]; 
 \draw[thick, red, rotate=60, rounded corners=8pt] (-0.02,1.8)--(-0.1, 2.15)--(-0.02,2.5);  
 \draw[thick, red, dashed, rotate=60, rounded corners=8pt]  (0.02,1.8)--(0.1, 2.15)--(0.02,2.5); 
 \draw[thick, dashed, blue, rotate=0, rounded corners=8pt] (-0.02,1.8)--(-0.1, 2.15)--(-0.02,2.5);  
 \draw[thick, blue, rotate=0, rounded corners=8pt]  (0.02,1.8)--(0.1, 2.15)--(0.02,2.5); 
  \node[scale=0.6, blue] at (-0.3,2.1) {$+$};
  \node[scale=0.6, red] at (0.5,1.0) {$-$};
  \node[scale=0.6, blue] at (1.1,0.0) {$+$};
  \node[scale=0.6, red] at (-2.0,0.9) {$-$};
  \node[scale=0.6, red] at (0.0,-1.1) {$-$};
  \node[scale=0.6, blue] at (-0.9,-0.7) {$+$};
  \end{scope}

\begin{scope} [xshift=8.7cm, yshift=-5.6cm]
 \draw[very thick, violet] (0,0) circle [radius=2.5cm];
 \draw[very thick, violet] (0,1.6) circle [radius=0.2cm]; 
 \draw[very thick, violet, rotate=60] (0,1.6) circle [radius=0.2cm]; 
 \draw[very thick, violet, rotate=120] (0,1.6) circle [radius=0.2cm];  
 \draw[very thick, violet, rotate=180] (0,1.6) circle [radius=0.2cm];  
 \draw[very thick, violet, rotate=-120] (0,1.6) circle [radius=0.2cm]; 
 \draw[very thick, violet, rotate=-60] (0,1.6) circle [radius=0.2cm]; 
 \draw[thick, red, rotate=-120, rounded corners=8pt] (-1.2,0.94)--(-0.8, 1.04)--(-0.18,1.5);  
 \draw[thick, red, dashed,  rotate=-120, rounded corners=8pt] (-1.2,0.98)--(-0.92, 1.28)--(-0.2,1.54);  
 \draw[thick, blue, rotate=180] (0,1.6) circle [radius=0.3cm]; 
\draw[thick, red, rotate=120] (0,1.6) circle [radius=0.3cm]; 
 \draw[thick, blue, rotate=-60, rounded corners=8pt] (-0.02,1.8)--(-0.1, 2.15)--(-0.02,2.5);  
 \draw[thick, blue, dashed, rotate=-60, rounded corners=8pt]  (0.02,1.8)--(0.1, 2.15)--(0.02,2.5); 
 \draw[thick, blue, rotate=60, rounded corners=8pt] (-0.02,1.8)--(-0.1, 2.15)--(-0.02,2.5);  
 \draw[thick, blue, dashed, rotate=60, rounded corners=8pt]  (0.02,1.8)--(0.1, 2.15)--(0.02,2.5); 
 \draw[thick, dashed, red, rotate=0, rounded corners=8pt] (-0.02,1.8)--(-0.1, 2.15)--(-0.02,2.5);  
 \draw[thick, red, rotate=0, rounded corners=8pt]  (0.02,1.8)--(0.1, 2.15)--(0.02,2.5); 
  \node[scale=0.6, red] at (-0.3,2.1) {$-$};
  \node[scale=0.6, blue] at (1.6,1.3) {$+$};
  \node[scale=0.6, red] at (1.1,0.0) {$-$};
  \node[scale=0.6, blue] at (-2.0,0.9) {$+$};
  \node[scale=0.6, blue] at (0.0,-1.1) {$+$};
  \node[scale=0.6, red] at (-0.9,-0.7) {$-$};
  \end{scope}

\begin{scope} [xshift=2.9cm, yshift=-5.6cm]
 \draw[very thick, violet] (0,0) circle [radius=2.5cm];
 \draw[very thick, violet] (0,1.6) circle [radius=0.2cm]; 
 \draw[very thick, violet, rotate=60] (0,1.6) circle [radius=0.2cm]; 
 \draw[very thick, violet, rotate=120] (0,1.6) circle [radius=0.2cm];  
 \draw[very thick, violet, rotate=180] (0,1.6) circle [radius=0.2cm];  
 \draw[very thick, violet, rotate=-120] (0,1.6) circle [radius=0.2cm]; 
 \draw[very thick, violet, rotate=-60] (0,1.6) circle [radius=0.2cm]; 
 \draw[thick, blue, rotate=180] (0,1.6) circle [radius=0.3cm]; 
\draw[thick, red, rotate=120] (0,1.6) circle [radius=0.3cm]; 
 \draw[thick, blue, rotate=-60, rounded corners=8pt] (-0.02,1.8)--(-0.1, 2.15)--(-0.02,2.5);  
 \draw[thick, blue, dashed, rotate=-60, rounded corners=8pt]  (0.02,1.8)--(0.1, 2.15)--(0.02,2.5); 
 \draw[thick, blue, rotate=60, rounded corners=8pt] (-0.02,1.8)--(-0.1, 2.15)--(-0.02,2.5);  
 \draw[thick, blue, dashed, rotate=60, rounded corners=8pt]  (0.02,1.8)--(0.1, 2.15)--(0.02,2.5); 
 \draw[thick, dashed, red, rotate=0, rounded corners=8pt] (-0.02,1.8)--(-0.1, 2.15)--(-0.02,2.5);  
 \draw[thick, red, rotate=0, rounded corners=8pt]  (0.02,1.8)--(0.1, 2.15)--(0.02,2.5); 
 \draw[thick, dashed, red, rotate=-120, rounded corners=8pt] (-0.02,1.8)--(-0.1, 2.15)--(-0.02,2.5);  
 \draw[thick, red, rotate=-120, rounded corners=8pt]  (0.02,1.8)--(0.1, 2.15)--(0.02,2.5); 
  \node[scale=0.6, red] at (-0.3,2.1) {$-$};
  \node[scale=0.6, blue] at (1.6,1.3) {$+$};
  \node[scale=0.6, red] at (2.0,-0.8) {$-$};
  \node[scale=0.6, blue] at (-2.0,0.9) {$+$};
  \node[scale=0.6, blue] at (0.0,-1.1) {$+$};
  \node[scale=0.6, red] at (-0.9,-0.7) {$-$};
\end{scope}

\begin{scope} [xshift=2.9cm, yshift=-11.2cm, rotate=-60]
 \draw[very thick, violet] (0,0) circle [radius=2.5cm];
 \draw[very thick, violet] (0,1.6) circle [radius=0.2cm]; 
 \draw[very thick, violet, rotate=60] (0,1.6) circle [radius=0.2cm]; 
 \draw[very thick, violet, rotate=120] (0,1.6) circle [radius=0.2cm];  
 \draw[very thick, violet, rotate=180] (0,1.6) circle [radius=0.2cm];  
 \draw[very thick, violet, rotate=-120] (0,1.6) circle [radius=0.2cm]; 
 \draw[very thick, violet, rotate=-60] (0,1.6) circle [radius=0.2cm]; 
 \draw[thick, blue, rotate=180] (0,1.6) circle [radius=0.3cm]; 
\draw[thick, red, rotate=120] (0,1.6) circle [radius=0.3cm]; 
 \draw[thick, blue, rotate=-60, rounded corners=8pt] (-0.02,1.8)--(-0.1, 2.15)--(-0.02,2.5);  
 \draw[thick, blue, dashed, rotate=-60, rounded corners=8pt]  (0.02,1.8)--(0.1, 2.15)--(0.02,2.5); 
 \draw[thick, blue, rotate=60, rounded corners=8pt] (-0.02,1.8)--(-0.1, 2.15)--(-0.02,2.5);  
 \draw[thick, blue, dashed, rotate=60, rounded corners=8pt]  (0.02,1.8)--(0.1, 2.15)--(0.02,2.5); 
 \draw[thick, dashed, red, rotate=0, rounded corners=8pt] (-0.02,1.8)--(-0.1, 2.15)--(-0.02,2.5);  
 \draw[thick, red, rotate=0, rounded corners=8pt]  (0.02,1.8)--(0.1, 2.15)--(0.02,2.5); 
 \draw[thick, dashed, red, rotate=-120, rounded corners=8pt] (-0.02,1.8)--(-0.1, 2.15)--(-0.02,2.5);  
 \draw[thick, red, rotate=-120, rounded corners=8pt]  (0.02,1.8)--(0.1, 2.15)--(0.02,2.5); 
  \node[scale=0.6, red] at (-0.3,2.1) {$-$};
  \node[scale=0.6, blue] at (1.6,1.3) {$+$};
  \node[scale=0.6, red] at (2.0,-0.8) {$-$};
  \node[scale=0.6, blue] at (-2.0,0.9) {$+$};
  \node[scale=0.6, blue] at (0.0,-1.1) {$+$};
  \node[scale=0.6, red] at (-0.9,-0.7) {$-$};
\end{scope}

\begin{scope} [xshift=8.7cm, yshift=-11.2cm]
 \draw[very thick, violet] (0,0) circle [radius=2.5cm];
 \draw[very thick, violet] (0,1.6) circle [radius=0.2cm]; 
 \draw[very thick, violet, rotate=60] (0,1.6) circle [radius=0.2cm]; 
 \draw[very thick, violet, rotate=120] (0,1.6) circle [radius=0.2cm];  
 \draw[very thick, violet, rotate=180] (0,1.6) circle [radius=0.2cm];  
 \draw[very thick, violet, rotate=-120] (0,1.6) circle [radius=0.2cm]; 
 \draw[very thick, violet, rotate=-60] (0,1.6) circle [radius=0.2cm]; 
 \draw[thick, blue, rotate=180] (0,1.6) circle [radius=0.3cm]; 
\draw[thick, red, rotate=60] (0,1.6) circle [radius=0.3cm]; 
 \draw[thick, blue, rotate=60, rounded corners=8pt] (-0.02,1.8)--(-0.1, 2.15)--(-0.02,2.5);  
 \draw[thick, blue, dashed, rotate=60, rounded corners=8pt]  (0.02,1.8)--(0.1, 2.15)--(0.02,2.5); 
 \draw[thick, dashed, red, rotate=180, rounded corners=8pt] (-0.02,1.8)--(-0.1, 2.15)--(-0.02,2.5);  
 \draw[thick, red, rotate=180, rounded corners=8pt]  (0.02,1.8)--(0.1, 2.15)--(0.02,2.5);  
  \node[scale=0.6, red] at (-0.3,-2.1) {$-$};
  \node[scale=0.6, blue] at (-2.0,0.9) {$+$};
  \node[scale=0.6, blue] at (0.0,-1.1) {$+$};
  \node[scale=0.6, red] at (-0.9,0.7) {$-$};
\end{scope}

	 \node[scale=0.8] at (-1.5,2.5) {$F_1$};
 	\node[scale=0.8] at (4.3,2.5) {$F_2$};
 	\node[scale=0.8] at (10.1,2.5) {$F_3$};
 	\node[scale=0.8] at (7.2,-3.1) {$F_4$};
 	\node[scale=0.8] at (1.4,-3.1) {$F_5$};
 	\node[scale=0.8] at (1.4,-8.7) {$F_6$};
 	\node[scale=0.8] at (7.2,-8.7) {$F_7$};
 	\draw[ ->, rounded corners=5pt] (2.4,1.5)--(2.9,1.6)-- (3.4,1.5);
 	\node[scale=0.8] at (2.9,2) {$R$};
 	\draw[ ->, rounded corners=5pt] (11.1,-2.7)--(11, -3.3)--(10.7,-3.7);
 	\node[scale=0.8] at (11.6,-3.2) {$F_3F_1$};
 	\draw[ ->, rounded corners=5pt] (5.7+0.97,-3.92)--(5.7+0.07,-3.84)--(5.7-0.83,-3.92) ;
 	\node[scale=0.8] at (5.7,-3.53) {$ \cdot (C_2A_3^{-1})$};
\end{tikzpicture}
 	\caption{Proof of Corollary~\ref{cor:gen1}.}
\end{figure}
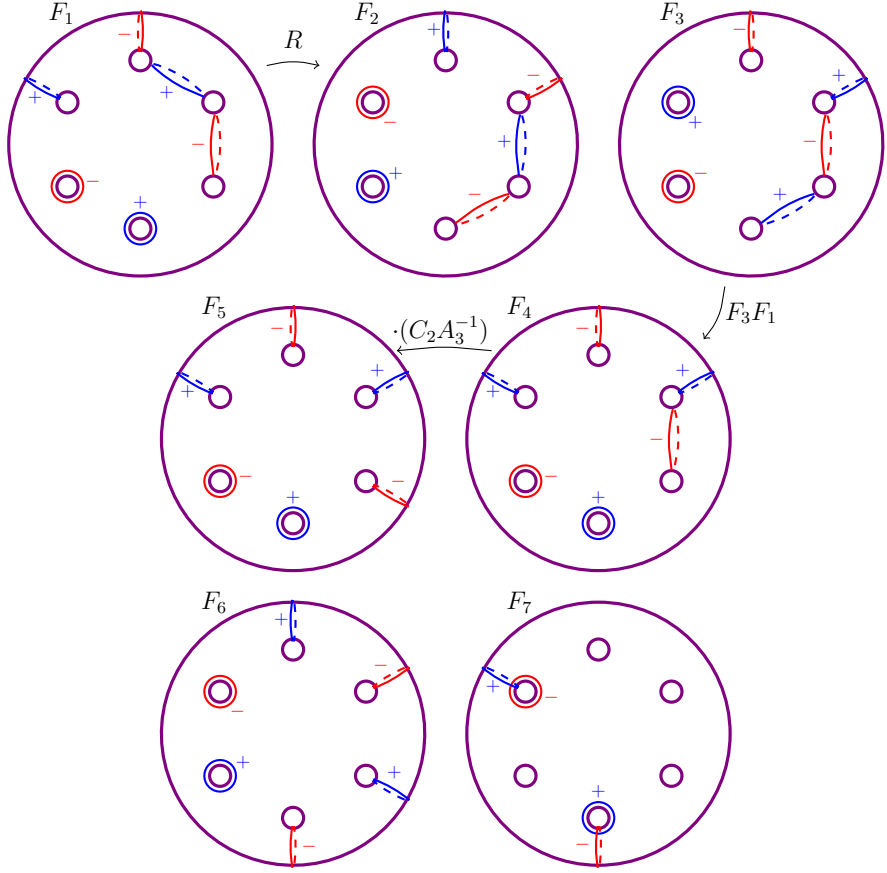

Then, $F_3F_1(c_3,b_6,a_2,a_1,b_5,c_2)=(b_4,a_6,a_2,a_1,b_5,c_2)$ so that\\
$F_4=B_4A_6A_2A_1^{-1}B_5^{-1}C_2^{-1}\in H$. Note that $F_3F_1(c_3) = b_4$ 
since \begin{eqnarray*}
t_{F_3F_1(c_3)}
&=& (F_3F_1)t_{c_3}(F_3F_1)^{-1}\\
&=& F_3F_1C_3F_1^{-1}F_3^{-1}\\
&=& C_3B_4C_3B_4^{-1}C_3^{-1}\\
&=& (t_{c_3}t_{b_4})t_{c_3}(t_{c_3}t_{b_4})^{-1}\\
&=& t_{t_{c_3}t_{b_4}(c_3)}\\
&=& t_{b_4}.
\end{eqnarray*}

$F_1F_4^{-1}=C_1A_2^{-1}$$\in$$H$ and then by conjugating $C_1A_2^{-1}$ with $R$ iteratively, we get $C_iA_{i+1}^{-1}$$\in$$H$ $\forall i$.

Let \[
F_5 = F_4(C_2A_3^{-1}) = B_4A_6A_2A_1^{-1}B_5^{-1}A_3^{-1}
\]
\[
F_6 = RF_5R^{-1} = B_5A_1A_3A_2^{-1}B_6^{-1}A_4^{-1}.
\]
%
%
%
%
and
\[ 
F_7 = F_5F_6 = B_4A_6B_6^{-1}A_4^{-1}.
\]

$(C_4A_5^{-1})F_7(c_4,a_5)=(b_4,a_5)$ so that $B_4A_5^{-1}$$\in$$H$ and then \mbox{$B_iA_{i+1}^{-1}$$\in$$H$ $\forall i$}.\\
$B_iC_i^{-1}=(B_iA_{i+1}^{-1})(A_{i+1}C_i^{-1})$$\in$$H$  $\forall i$.

$(A_4B_3^{-1})F_7(a_4,b_3)=(b_4,b_3)$ so that $B_4B_3^{-1}$$\in$$H$ and then \mbox{$B_{i+1}B_i^{-1}$$\in$$H$ $\forall i$}.
In particular, $B_1B_2^{-1}$$\in$$H$.

$C_1C_2^{-1}=(C_1B_1^{-1})(B_1B_2^{-1})(B_2C_2^{-1})$$\in$$H$.

$A_1A_2^{-1}=(A_1B_6^{-1})(B_6B_1^{-1})(B_1A_2^{-1})$$\in$$H$.

It follows from  Theorem~\ref{thm:thmKorkmaz} that
$H=\mod(\Sigma_6)$, completing the proof of the corollary.
\end{proof}

\begin{corollary} \label{cor:gen2} 
If $g=7$, then the mapping class group $\mod(\Sigma_g)$ is generated by two elements
$R$ and $C_1B_4A_6A_7^{-1}B_5^{-1}C_2^{-1}$.
\end{corollary}

\begin{figure}
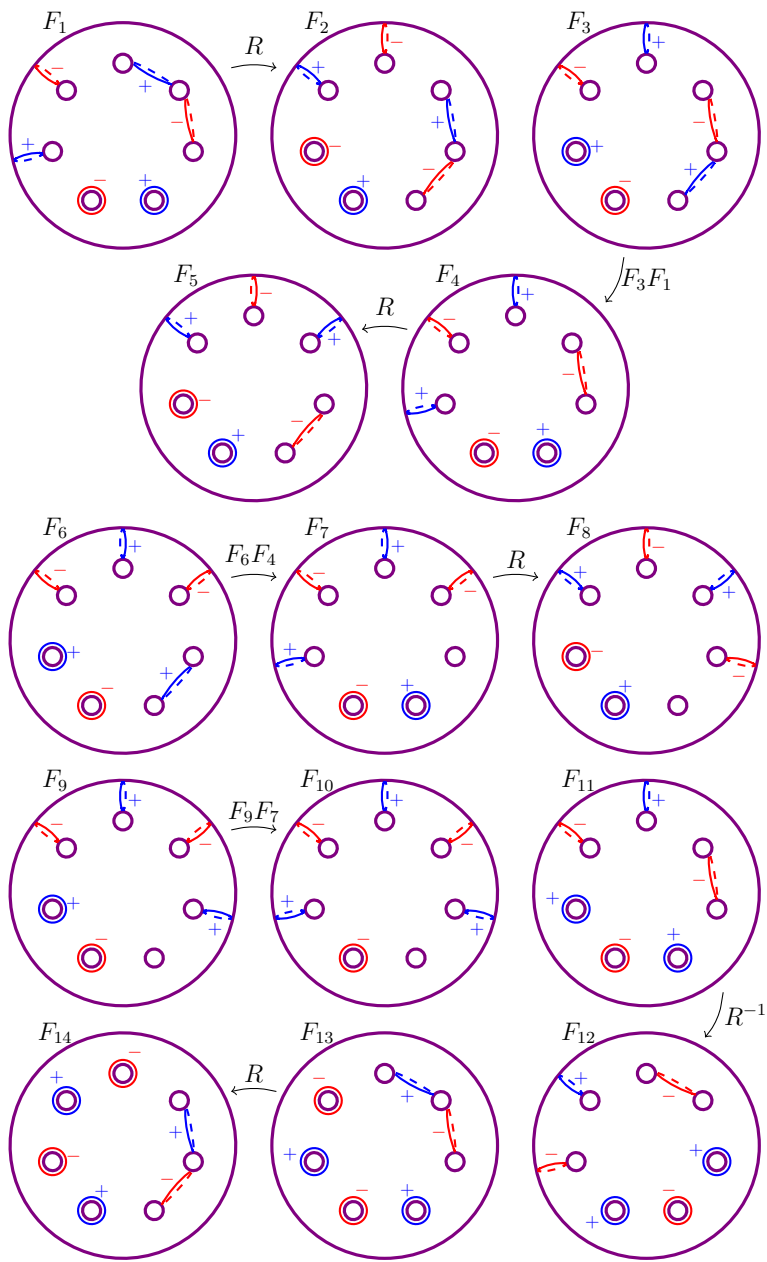


 	\caption{Proof of Corollary~\ref{cor:gen2}.}
\end{figure}

\begin{proof}
Let $F_1=C_1B_4A_6A_7^{-1}B_5^{-1}C_2^{-1}$. Let us denote by
$H$ the subgroup of $\mod(\Sigma_7)$ generated by the set
$\{ R, F_1\}$.

%
Let
\[
F_2 = RF_1R^{-1} = C_2B_5A_7A_1^{-1}B_6^{-1}C_3^{-1}
\]
and \[
F_3 = F_2^{-1} = C_3B_6A_1A_7^{-1}B_5^{-1}C_2^{-1}.
\]

Then, $F_3F_1(c_3,b_6,a_1,a_7,b_5,c_2)=(b_4,a_6,a_1,a_7,b_5,c_2)$ so that\\
$F_4=B_4A_6A_1A_7^{-1}B_5^{-1}C_2^{-1}\in H$.

Let \[
F_5 = RF_4R^{-1} = B_5A_7A_2A_1^{-1}B_6^{-1}C_3^{-1}
\]
and \[
F_6 = F_5^{-1} = C_3B_6A_1A_2^{-1}A_7^{-1}B_5^{-1}.
\]

Then, $F_6F_4(c_3,b_6,a_1,a_2,a_7,b_5)=(b_4,a_6,a_1,a_2,a_7,b_5)$ so that\\
$F_7=B_4A_6A_1A_2^{-1}A_7^{-1}B_5^{-1}$$\in$$H$.

Let \[
F_8 = RF_7R^{-1} = B_5A_7A_2A_3^{-1}A_1^{-1}B_6^{-1}
\]
and \[ 
F_9 = F_8^{-1} = B_6A_1A_3A_2^{-1}A_7^{-1}B_5^{-1}.
\]

$F_9F_7(b_6,a_1,a_3,a_2,a_7,b_5)=(a_6,a_1,a_3,a_2,a_7,b_5)$ so that\\
$F_{10}=A_6A_1A_3A_2^{-1}A_7^{-1}B_5^{-1}$$\in$$H$.

$F_{10}F_8=A_6B_6^{-1}$$\in$$H$ and then by conjugating $A_6B_6^{-1}$ with $R$ iteratively, we get $A_iB_i^{-1}$$\in$$H$ $\forall i$.

Let \[
F_{11} = (B_6A_6^{-1})F_4 = B_4B_6A_1A_7^{-1}B_5^{-1}C_2^{-1}
\]
and \[ 
F_{12} = R^{-1}F_{11}R = B_3B_5A_7A_6^{-1}B_4^{-1}C_1^{-1}.
\]
$F_{12}F_1=B_3C_2^{-1}$$\in$$H$ and then $B_{i+1}C_i^{-1}$$\in$$H$ $\forall i$.

Let \[
F_{13} = (B_6A_6^{-1})F_1(A_7B_7^{-1}) = C_1B_4B_6B_7^{-1}B_5^{-1}C_2^{-1}
\]
and \[ 
F_{14} = RF_{13}R^{-1} = C_2B_5B_7B_1^{-1}B_6^{-1}C_3^{-1}.
\]
%
%
%

Then  $F_{13}F_{14}(C_3B_4^{-1})=C_1B_1^{-1}$$\in$$H$ and then $C_iB_i^{-1}$$\in$$H$ $\forall i$.

It follows from Corollary~\ref{cor:thmKorkmaz} that
$H=\mod(\Sigma_7)$, completing the proof of the corollary.
\end{proof}

\begin{corollary} \label{cor:gen3} 
If $g=8$, then the mapping class group $\mod(\Sigma_g)$ is generated by two elements
$R$ and $B_1C_4A_7A_8^{-1}C_5^{-1}B_2^{-1}$.
\end{corollary}
\begin{figure}
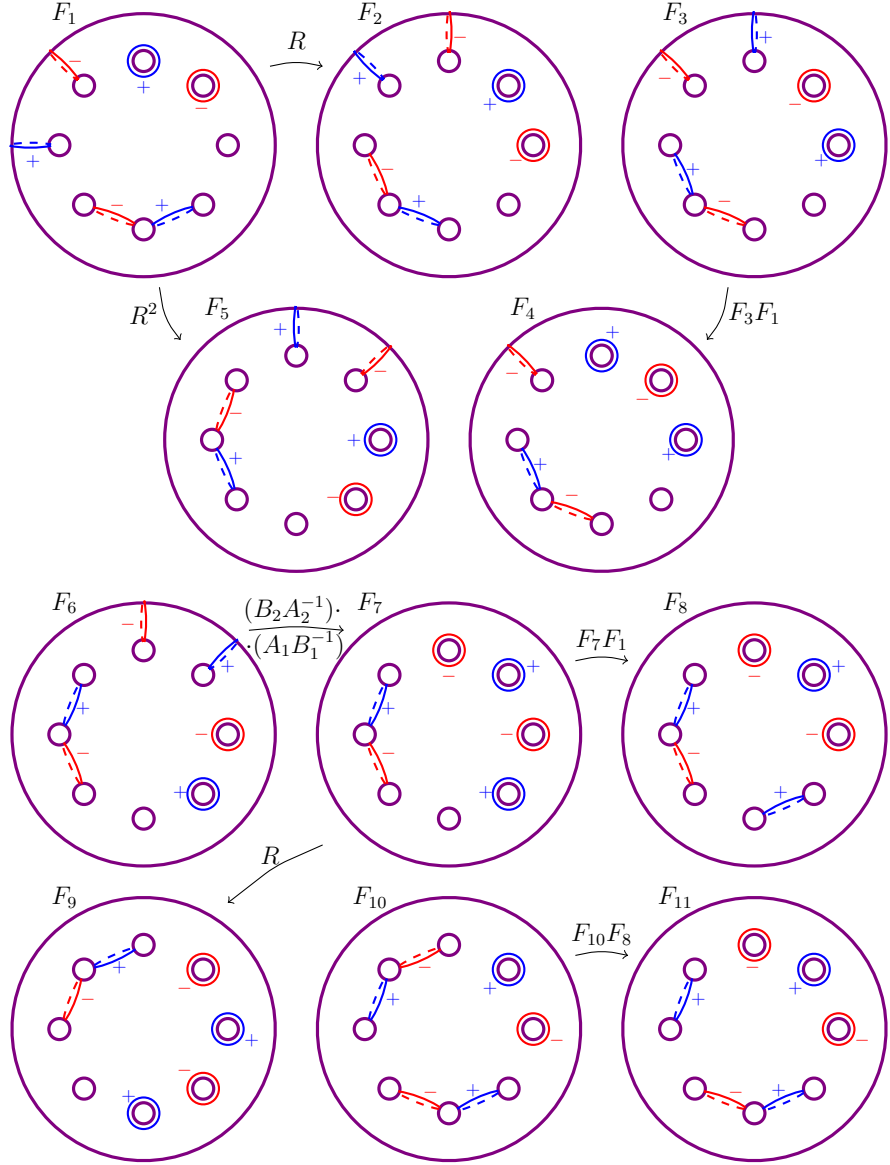


 	\caption{Proof of Corollary~\ref{cor:gen3}.}
\end{figure}

\begin{proof}

Let $F_1=B_1C_4A_7A_8^{-1}C_5^{-1}B_2^{-1}$. Let us denote by
$H$ the subgroup of $\mod(\Sigma_8)$ generated by the set
$\{ R, F_1\}$.

Let
\[ 
F_2 = RF_1R^{-1} = B_2C_5A_8A_1^{-1}C_6^{-1}B_3^{-1}
\]
and \[
F_3 = F_2^{-1} = B_3C_6A_1A_8^{-1}C_5^{-1}B_2^{-1}.\]

Then, $F_3F_1(b_3,c_6,a_1,a_8,c_5,b_2)=(b_3,c_6,b_1,a_8,c_5,b_2)$ so that\\
$F_4=B_3C_6B_1A_8^{-1}C_5^{-1}B_2^{-1}\in H$.

$F_4F_3^{-1}=B_1A_1^{-1}$$\in$$H$ and then by conjugating $B_1A_1^{-1}$ with $R$ iteratively, we get $B_iA_i^{-1}$$\in$$H$  $\forall i$.

Let \[
F_5 = R^2F_1R^{-2} = B_3C_6A_1A_2^{-1}C_7^{-1}B_4^{-1}
\]
\[
F_6 = F_5^{-1} = B_4C_7A_2A_1^{-1}C_6^{-1}B_3^{-1}\]
and \[
F_7 = (B_2A_2^{-1})F_6(A_1B_1^{-1}) = B_4C_7B_2B_1^{-1}C_6^{-1}B_3^{-1}.\]

Then, $F_7F_1(b_4,c_7,b_2,b_1,c_6,b_3)=(c_4,c_7,b_2,b_1,c_6,b_3)$ so that\\
$F_8=C_4C_7B_2B_1^{-1}C_6^{-1}B_3^{-1}$$\in$$H$.

$F_8F_7^{-1}=C_4B_4^{-1}$$\in$$H$ and then we get $C_iB_i^{-1}$$\in$$H$  $\forall i$.

Let \[
F_9 = RF_7R^{-1} = B_5C_8B_3B_2^{-1}C_7^{-1}B_4^{-1} \]
and \[
F_{10} = (C_4B_4^{-1})F_9^{-1}(B_5C_5^{-1}) = C_4C_7B_2B_3^{-1}C_8^{-1}C_5^{-1}.
\]

Then, $F_{10}F_8(c_4,c_7,b_2,b_3,c_8,c_5)=(c_4,c_7,b_2,b_3,b_1,c_5)$ so that\\
$F_{11}=C_4C_7B_2B_3^{-1}B_1^{-1}C_5^{-1}$$\in$$H$.

$F_{10}^{-1}F_{11}=C_8B_1^{-1}$$\in$$H$ and then we get $C_iB_{i+1}^{-1}$$\in$$H$  $\forall i$.

It follows from Corollary~\ref{cor:thmKorkmaz} that
$H=\mod(\Sigma_8)$, completing the proof of the corollary.
\end{proof}
 
\begin{corollary} \label{cor:gen4} 
If $g\geq 9$, then the mapping class group $\mod(\Sigma_g)$ is generated by two elements
$R$ and $C_1B_4A_7A_8^{-1}B_5^{-1}C_2^{-1}$.
\end{corollary}

\begin{figure}
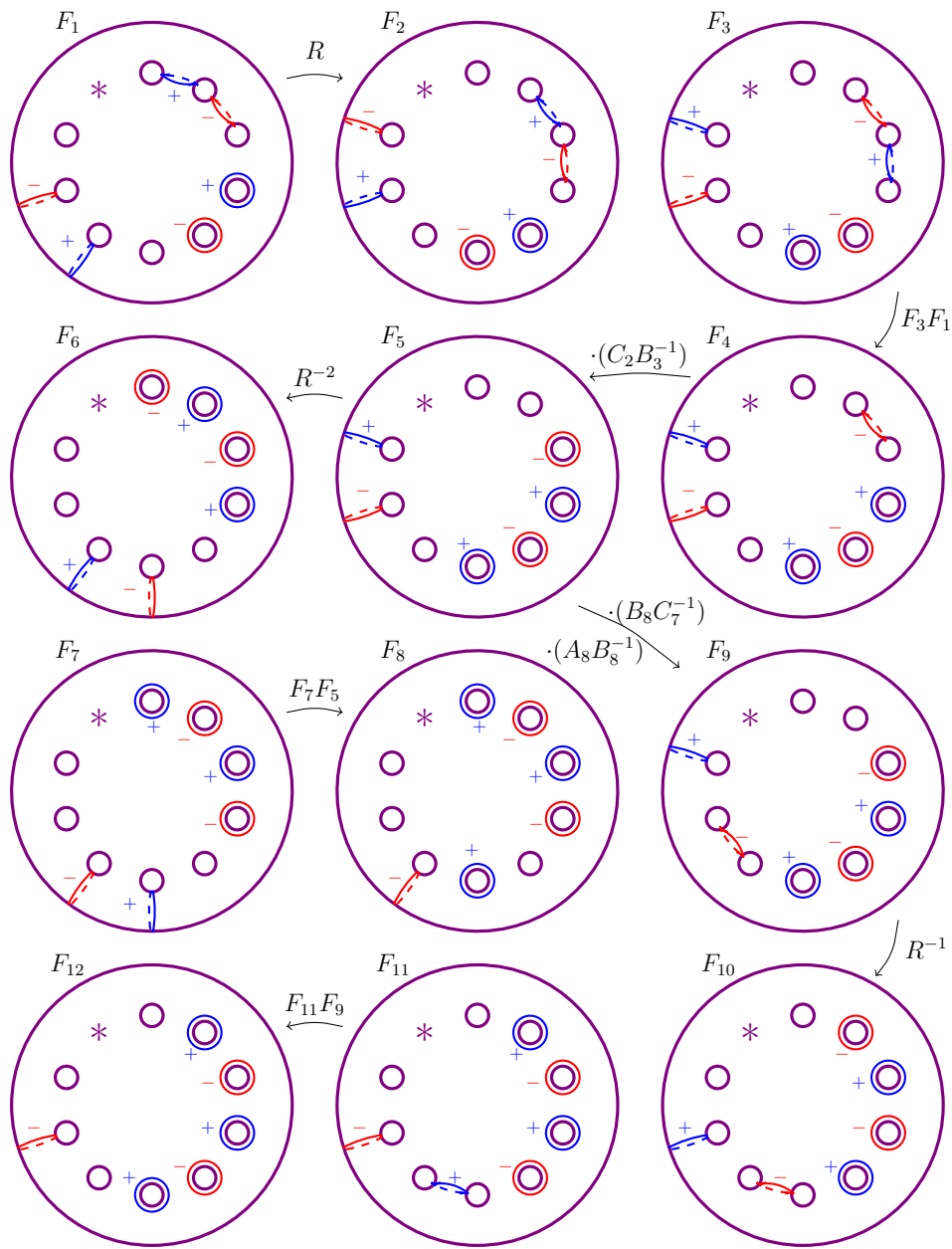


 	\caption{Proof of Corollary~\ref{cor:gen4}.}
\end{figure}

\begin{proof}

Let $F_1=C_1B_4A_7A_8^{-1}B_5^{-1}C_2^{-1}$. Let us denote by
$H$ the subgroup of $\mod(\Sigma_g)$ generated by the set
$\{ R, F_1\}$.

%
Let
\[
F_2 = RF_1R^{-1} = C_2B_5A_8A_9^{-1}B_6^{-1}C_3^{-1}
\]
and \[
F_3 = F_2^{-1} = C_3B_6A_9A_8^{-1}B_5^{-1}C_2^{-1}.
\]

Then, $F_3F_1(c_3,b_6,a_9,a_8,b_5,c_2)=(b_4,b_6,a_9,a_8,b_5,c_2)$ so that\\
$F_4=B_4B_6A_9A_8^{-1}B_5^{-1}C_2^{-1}\in H$.

$F_4F_3^{-1}=B_4C_3^{-1}$$\in$$H$ and then by conjugating $B_4C_3^{-1}$ with $R$ iteratively, we get $B_{i+1}C_i^{-1}$$\in$$H$ $\forall i$.

Let \[
F_5 = F_4(C_2B_3^{-1}) = B_4B_6A_9A_8^{-1}B_5^{-1}B_3^{-1}
\]
\[
F_6 = R^{-2}F_5R^2 = B_2B_4A_7A_6^{-1}B_3^{-1}B_1^{-1}
\]
and \[
F_7 = F_6^{-1} = B_1B_3A_6A_7^{-1}B_4^{-1}B_2^{-1}. 
\]

Then, $F_7F_5(b_1,b_3,a_6,a_7,b_4,b_2)=(b_1,b_3,b_6,a_7,b_4,b_2)$ so that\\
$F_8=B_1B_3B_6A_7^{-1}B_4^{-1}B_2^{-1}$$\in$$H$.

Then, $F_8F_7^{-1}= B_6A_6^{-1}$$\in$$H$ and then \mbox{$B_iA_i^{-1}$$\in$$H$ $\forall i$}.

Let \[
F_9 = F_5(A_8B_8^{-1})(B_8C_7^{-1}) = B_4B_6A_9C_7^{-1}B_5^{-1}B_3^{-1}
\]
\[
F_{10} = R^{-1}F_9R = B_3B_5A_8C_6^{-1}B_4^{-1}B_2^{-1}
\]
and \[
F_{11} = F_{10}^{-1} = B_2B_4C_6A_8^{-1}B_5^{-1}B_3^{-1}. 
\]

Then, $F_{11}F_9(b_2,b_4,c_6,a_8,b_5,b_3)=(b_2,b_4,b_6,a_8,b_5,b_3)$ so that\\
$F_{12}=B_2B_4B_6A_8^{-1}B_5^{-1}B_3^{-1}$$\in$$H$.

Then, $F_{12}F_{11}^{-1}= B_6C_6^{-1}$$\in$$H$ and then \mbox{$B_iC_i^{-1}$$\in$$H$ $\forall i$}.

It follows from Corollary~\ref{cor:thmKorkmaz} that
$H=\mod(\Sigma_g)$, completing the proof of the corollary.
\end{proof}

\begin{corollary} \label{cor:gen12} 
If $g=8$, then the mapping class group $\mod(\Sigma_g)$ is generated by two elements
$R$ and $B_1A_5C_5C_7^{-1}A_7^{-1}B_3^{-1}$.
\end{corollary}

\begin{proof}

Let $F_1=B_1A_5C_5C_7^{-1}A_7^{-1}B_3^{-1}$. Let us denote by
$H$ the subgroup of $\mod(\Sigma_g)$ generated by the set
$\{ R, F_1\}$.

%
Let
\[
F_2 = RF_1R^{-1} = B_2A_6C_6C_8^{-1}A_8^{-1}B_4^{-1}
\]
and \[
F_3 = F_2^{-1} = B_4A_8C_8C_6^{-1}A_6^{-1}B_2^{-1}.
\]

Then, $F_3F_1(b_4,a_8,c_8,c_6,a_6,b_2)=(b_4,a_8,b_1,c_6,a_6,b_2)$ so that\\
$F_4=B_4A_8B_1C_6^{-1}A_6^{-1}B_2^{-1}\in H$.

$F_4F_3^{-1}=B_1C_8^{-1}$$\in$$H$ and then by conjugating $B_1C_8^{-1}$ with $R$ iteratively, we get $B_{i+1}C_i^{-1}$$\in$$H$ $\forall i$.

Let \[
F_5 = RF_4R^{-1} = B_5A_1B_2C_7^{-1}A_7^{-1}B_3^{-1}. 
\]

Then, $F_5F_4(b_5,a_1,b_2,c_7,a_7,b_3)=(b_5,b_1,b_2,c_7,a_7,b_3)$ so that\\
$F_6=B_5B_1B_2C_7^{-1}A_7^{-1}B_3^{-1}$$\in$$H$.

Then, $F_6F_5^{-1}= B_1A_1^{-1}$$\in$$H$ and then \mbox{$B_iA_i^{-1}$$\in$$H$ $\forall i$}.

Let \[
F_7 = (C_4B_5^{-1})F_6(C_7B_8^{-1})(A_7B_7^{-1}) = C_4B_1B_2B_3^{-1}B_8^{-1}B_7^{-1}
\]
\[
F_8 = RF_9R^{-1} = C_5B_2B_3B_4^{-1}B_1^{-1}B_8^{-1}
\]
and \[
F_9 = F_8^{-1} = B_8B_1B_4B_3^{-1}B_2^{-1}C_5^{-1}. 
\]

Then, $F_9F_7(b_8,b_1,b_4,b_3,b_2,c_5)=(b_8,b_1,c_4,b_3,b_2,c_5)$ so that\\
$F_{10}=B_8B_1C_4B_3^{-1}B_2^{-1}C_5^{-1}$$\in$$H$.

Then, $F_{10}F_9^{-1}= C_4B_4^{-1}$$\in$$H$ and then \mbox{$C_iB_i^{-1}$$\in$$H$ $\forall i$}.

It follows from Corollary~\ref{cor:thmKorkmaz} that
$H=\mod(\Sigma_8)$, completing the proof of the corollary.
\end{proof}

\begin{corollary} \label{cor:gen13} 
If $g=9$, then the mapping class group $\mod(\Sigma_g)$ is generated by two elements
$R$ and $B_1A_3C_5C_8^{-1}A_6^{-1}B_4^{-1}$.
\end{corollary}

\begin{proof}
Let $F_1=B_1A_3C_5C_8^{-1}A_6^{-1}B_4^{-1}$. Let us denote by
$H$ the subgroup of $\mod(\Sigma_9)$ generated by the set
$\{ R, F_1\}$.

%
Let
\[
F_2 = RF_1R^{-1} = B_2A_4C_6C_9^{-1}A_7^{-1}B_5^{-1}
\]
and \[
F_3 = F_2^{-1} = B_5A_7C_9C_6^{-1}A_4^{-1}B_2^{-1}.
\]

Then, $F_3F_1(b_5,a_7,c_9,c_6,a_4,b_2)=(c_5,a_7,b_1,c_6,b_4,b_2)$ so that\\
$F_4=C_5A_7B_1C_6^{-1}B_4^{-1}B_2^{-1}\in H$.

Let \[
F_5 = RF_4R^{-1} = C_6A_8B_2C_7^{-1}B_5^{-1}B_3^{-1}
\]
and \[
F_6 = F_5^{-1} = B_3B_5C_7B_2^{-1}A_8^{-1}C_6^{-1}. 
\]

Then, $F_6F_4(b_3,b_5,c_7,b_2,a_8,c_6)=(b_3,c_5,c_7,b_2,a_8,c_6)$ so that\\
$F_7=B_3C_5C_7B_2^{-1}A_8^{-1}C_6^{-1}\in H$.

$F_7F_6^{-1}=C_5B_5^{-1}$$\in$$H$ and then by conjugating $C_5B_5^{-1}$ with $R$ iteratively, we get $C_iB_i^{-1}$$\in$$H$ $\forall i$.

Let \[
F_8 = (B_7C_7^{-1})F_6(C_6B_6^{-1}) = B_3B_5B_7B_2^{-1}A_8^{-1}B_6^{-1}
\]
\[
F_9 = RF_8R^{-1} = B_4B_6B_8B_3^{-1}A_9^{-1}B_7^{-1}
\]
and \[
F_{10} = F_9^{-1} = B_7A_9B_3B_8^{-1}B_6^{-1}B_4^{-1}. 
\]

Then, $F_{10}F_8(b_7,a_9,b_3,b_8,b_6,b_4)=(b_7,a_9,b_3,a_8,b_6,b_4)$ so that\\
$F_{11}=B_7A_9B_3A_8^{-1}B_6^{-1}B_4^{-1}$$\in$$H$.

Then, $F_{11}^{-1}F_{10}= A_8B_8^{-1}$$\in$$H$ and then \mbox{$A_iB_i^{-1}$$\in$$H$ $\forall i$}.

Then, $F_4(B_4A_4^{-1})F_2(B_5C_5^{-1})= B_1C_9^{-1} $$\in$$H$ and then \mbox{$B_{i+1}C_i^{-1}$$\in$$H$ $\forall i$}.

It follows from Corollary~\ref{cor:thmKorkmaz} that
$H=\mod(\Sigma_9)$, completing the proof of the corollary.
\end{proof}

\begin{corollary} \label{cor:gen5} 
If $g=10$, then the mapping class group $\mod(\Sigma_g)$ is generated by two elements
$R$ and $A_1C_1B_3B_7^{-1}C_5^{-1}A_5^{-1}$.
\end{corollary}

\begin{proof}

Let $F_1=A_1C_1B_3B_7^{-1}C_5^{-1}A_5^{-1}$. Let us denote by
$H$ the subgroup of $\mod(\Sigma_{10})$ generated by the set
$\{ R, F_1\}$.

%
Let
\[
F_2 = RF_1R^{-1} = A_2C_2B_4B_8^{-1}C_6^{-1}A_6^{-1}.
\]

Then, $F_2F_1(a_2,c_2,b_4,b_8,c_6,a_6)=(a_2,b_3,b_4,b_8,b_7,a_6)$ so that\\
$F_3=A_2B_3B_4B_8^{-1}B_7^{-1}A_6^{-1}\in H$.

Let \[
F_4 = R^4F_3R^{-4} = A_6B_7B_8B_2^{-1}B_1^{-1}A_{10}^{-1}
\]
and \[
F_5 = F_4^{-1} = A_{10}B_1B_2B_8^{-1}B_7^{-1}A_6^{-1}.
\]

Then, $F_5F_3(a_{10},b_1,b_2,b_8,b_7,a_6)=(a_{10},b_1,a_2,b_8,b_7,a_6)$ so that\\
$F_6=A_{10}B_1A_2B_8^{-1}B_7^{-1}A_6^{-1}$$\in$$H$.

$F_6F_5^{-1}=A_2B_2^{-1}$$\in$$H$ and then by conjugating $A_2B_2^{-1}$ with $R$ iteratively, we get $A_iB_i^{-1}$$\in$$H$ $\forall i$.

Let \[
F_7 = (B_2A_2^{-1})(A_3B_3^{-1})F_3(B_7A_7^{-1})(A_6B_6^{-1}) = B_2A_3B_4B_8^{-1}A_7^{-1}B_6^{-1}
\]
\[
F_8 = RF_2F_3^{-1}R^{-1}F_7 = B_2A_3C_3C_7^{-1}A_7^{-1}B_6^{-1}
\]
\[
F_9 = F_8^{-1} = B_6A_7C_7C_3^{-1}A_3^{-1}B_2^{-1}
\]
and \[
F_{10} = R^4F_9R^{-4} = B_{10}A_1C_1C_7^{-1}A_7^{-1}B_6^{-1}.
\]

Then, $F_{10}F_8(b_{10},a_1,c_1,c_7,a_7,b_6)=(b_{10},a_1,b_2,c_7,a_7,b_6)$ so that\\
$F_{11}=B_{10}A_1B_2C_7^{-1}A_7^{-1}B_6^{-1}$$\in$$H$.

Then, $F_{11}F_{10}^{-1}= B_2C_1^{-1}$$\in$$H$ and then \mbox{$B_{i+1}C_i^{-1}$$\in$$H$ $\forall i$}.

Let \[
F_{12} =(B_2A_2^{-1})F_3(A_6B_6^{-1})(B_6C_5^{-1})(B_7A_7^{-1})(B_8A_8^{-1})= B_2B_3B_4A_8^{-1}A_7^{-1}C_5^{-1}
\]
\[
F_{13} = F_{12}^{-1} = C_5A_7A_8B_4^{-1}B_3^{-1}B_2^{-1}
\]
and \[
F_{14} = RF_{13}R^{-1} = C_6A_8A_9B_5^{-1}B_4^{-1}B_3^{-1}. 
\]

Then, $F_{14}F_{12}(c_6,a_8,a_9,b_5,b_4,b_3)=(c_6,a_8,a_9,c_5,b_4,b_3)$ so that\\
$F_{15}=C_6A_8A_9C_5^{-1}B_4^{-1}B_3^{-1}$$\in$$H$.

Then, $F_{15}^{-1}F_{14}= C_5B_5^{-1}$$\in$$H$ and then \mbox{$C_iB_i^{-1}$$\in$$H$ $\forall i$}.

It follows from Corollary~\ref{cor:thmKorkmaz} that
$H=\mod(\Sigma_{10})$, completing the proof of the corollary.
\end{proof}

\begin{corollary} \label{cor:gen6} 
If $g \geq 13$, then the mapping class group $\mod(\Sigma_g)$ is generated by two elements
$R$ and $A_1B_4C_8C_{10}^{-1}B_6^{-1}A_3^{-1}$.
\end{corollary}

\begin{proof}
Let $F_1=A_1B_4C_8C_{10}^{-1}B_6^{-1}A_3^{-1}$. Let us denote by
$H$ the subgroup of $\mod(\Sigma_g)$ generated by the set
$\{ R, F_1\}$.

%
Let
\[
F_2 = RF_1R^{-1} = A_2B_5C_9C_{11}^{-1}B_7^{-1}A_4^{-1}
\]
and \[
F_3 = F_2^{-1} = A_4B_7C_{11}C_9^{-1}B_5^{-1}A_2^{-1}.
\]

Then, $F_3F_1(a_4,b_7,c_{11},c_9,b_5,a_2)=(b_4,b_7,c_{11},c_9,b_5,a_2)$ so that\\
$F_4=B_4B_7C_{11}C_9^{-1}B_5^{-1}A_2^{-1}\in H$.

$F_4F_3^{-1}=B_4A_4^{-1}$$\in$$H$ and then by conjugating $B_4A_4^{-1}$ with $R$
iteratively, we get $B_iA_i^{-1}$$\in$$H$ $\forall i$.

Let \[
F_5 = R^2F_1R^{-2} = A_3B_6C_{10}C_{12}^{-1}B_8^{-1}A_5^{-1}
\]
and \[
F_6 = F_5^{-1} = A_5B_8C_{12}C_{10}^{-1}B_6^{-1}A_3^{-1}.
\]

Then, $F_6F_1(a_5,b_8,c_{12},c_{10},b_6,a_3)=(a_5,c_8,c_{12},c_{10},b_6,a_3)$ so that\\
$F_7=A_5C_8C_{12}C_{10}^{-1}B_6^{-1}A_3^{-1}$$\in$$H$.

Then, $F_7F_6^{-1}= C_8B_8^{-1}$$\in$$H$ and then
\mbox{$C_iB_i^{-1}$$\in$$H$ $\forall i$}.

Let \[
F_8 = (A_4B_4^{-1})F_1(A_3B_3^{-1}) = A_1A_4C_8C_{10}^{-1}B_6^{-1}B_3^{-1}
\]
\[
F_9 = R^3F_8R^{-3} = A_4A_7C_{11}C_{13}^{-1}B_9^{-1}B_6^{-1}
\]
and \[
F_{10} = F_9^{-1} = B_6B_9C_{13}C_{11}^{-1}A_7^{-1}A_4^{-1}.
\]

Then, $F_{10}F_8(b_6,b_9,c_{13},c_{11},a_7,a_4)=(b_6,c_8,c_{13},c_{11},a_7,a_4)$ so
that\\
$F_{11}=B_6C_8C_{13}C_{11}^{-1}A_7^{-1}A_4^{-1}$$\in$$H$.

Then, $F_{11}F_{10}^{-1}= C_8B_9^{-1}$$\in$$H$ and then
\mbox{$C_iB_{i+1}^{-1}$$\in$$H$ $\forall i$}.

It follows from Corollary~\ref{cor:thmKorkmaz} that
$H=\mod(\Sigma_g)$, completing the proof of the corollary.
\end{proof}

\begin{corollary} \label{cor:gen7} 
If $g \geq 12$, then the mapping class group $\mod(\Sigma_g)$ is generated by two elements
$R$ and $B_1A_3C_6C_{10}^{-1}A_7^{-1}B_5^{-1}$.
\end{corollary}

\begin{proof}
Let $F_1=B_1A_3C_6C_{10}^{-1}A_7^{-1}B_5^{-1}$. Let us denote by
$H$ the subgroup of $\mod(\Sigma_g)$ generated by the set
$\{ R, F_1\}$.

%
Let
\[
F_2 = RF_1R^{-1} = B_2A_4C_7C_{11}^{-1}A_8^{-1}B_6^{-1}
\]
and \[
F_3 = F_2^{-1} = B_6A_8C_{11}C_7^{-1}A_4^{-1}B_2^{-1}.
\]

Then, $F_3F_1(b_6,a_8,c_{11},c_7,a_4,b_2)=(c_6,a_8,c_{11},c_7,a_4,b_2)$ so that\\
$F_4=C_6A_8C_{11}C_7^{-1}A_4^{-1}B_2^{-1}\in H$.

$F_4F_3^{-1}=C_6B_6^{-1}$$\in$$H$ and then by conjugating $C_6B_6^{-1}$ with $R$
iteratively, we get $C_iB_i^{-1}$$\in$$H$ $\forall i$.

Let \[
F_5 = F_1(C_{10}B_{10}^{-1})(B_5C_5^{-1}) = B_1A_3C_6B_{10}^{-1}A_7^{-1}C_5^{-1}
\]
and \[
F_6 = R^2F_5R^{-2} = B_3A_5C_8B_{12}^{-1}A_9^{-1}C_7^{-1}.
\]

Then, $F_6F_5(b_3,a_5,c_8,b_{12},a_9,c_7)=(a_3,a_5,c_8,b_{12},a_9,c_7)$ so that\\
$F_7=A_3A_5C_8B_{12}^{-1}A_9^{-1}C_7^{-1}$$\in$$H$.

Then, $F_7F_6^{-1}= A_3B_3^{-1}$$\in$$H$ and then
\mbox{$A_iB_i^{-1}$$\in$$H$ $\forall i$}.

Let \[
F_8 = (C_1B_1^{-1})(B_3A_3^{-1})F_1(B_5C_5^{-1}) = C_1B_3C_6C_{10}^{-1}A_7^{-1}C_5^{-1}
\]
and \[
F_9 = RF_8R^{-1} = C_2B_4C_7C_{11}^{-1}A_8^{-1}C_6^{-1}.
\]

Then, $F_9F_8(c_2,b_4,c_7,c_{11},a_8,c_6)=(b_3,b_4,c_7,c_{11},a_8,c_6)$ so that\\
$F_{10}=B_3B_4C_7C_{11}^{-1}A_8^{-1}C_6^{-1}$$\in$$H$.

Then, $F_{10}F_9^{-1}= B_3C_2^{-1}$$\in$$H$ and then
\mbox{$B_{i+1}C_i^{-1}$$\in$$H$ $\forall i$}.

It follows from Corollary~\ref{cor:thmKorkmaz} that
$H=\mod(\Sigma_g)$, completing the proof of the corollary.
\end{proof}

\begin{lemma} \label{lem:gen9} 
If $g \geq 11$, then the mapping class group $\mod(\Sigma_g)$ is generated 
by two elements $R$ and $A_1B_2C_4C_{g-1}^{-1}B_{g-3}^{-1}A_{g-4}^{-1}$.
\end{lemma}

\begin{proof}
Let $F_1=A_1B_2C_4C_{g-1}^{-1}B_{g-3}^{-1}A_{g-4}^{-1}$. Let us denote by
$H$ the subgroup of $\mod(\Sigma_g)$ generated by the set
$\{ R, F_1\}$.

%
Let
\[
F_2 = RF_1R^{-1} = A_2B_3C_5C_g^{-1}B_{g-2}^{-1}A_{g-3}^{-1}.
\]

Then, $F_2F_1(a_2,b_3,c_5,c_g,b_{g-2},a_{g-3})=(b_2,b_3,c_5,c_g,b_{g-2},b_{g-3})$ so
that\\
$F_3=B_2B_3C_5C_g^{-1}B_{g-2}^{-1}B_{g-3}^{-1}\in H$.

Let \[
F_4 = R^{-1}F_3R = B_1B_2C_4C_{g-1}^{-1}B_{g-3}^{-1}B_{g-4}^{-1}
\]
and \[
F_5 = F_3^{-1} = B_{g-3}B_{g-2}C_gC_5^{-1}B_3^{-1}B_2^{-1}.
\]

Then, $F_5F_4(b_{g-3},b_{g-2},c_g,c_5,b_3,b_2)=(b_{g-3},b_{g-2},b_1,c_5,b_3,b_2)$ so
that\\
$F_6=B_{g-3}B_{g-2}B_1C_5^{-1}B_3^{-1}B_2^{-1}$$\in$$H$.

$F_6F_5^{-1}=B_1C_g^{-1}$$\in$$H$ and then by conjugating $B_1C_g^{-1}$ with $R$
iteratively, we get $B_{i+1}C_i^{-1}$$\in$$H$ $\forall i$.

Let \[
F_7 = (C_{g-3}B_{g-2}^{-1})(C_{g-4}B_{g-3}^{-1})F_6 = C_{g-3}C_{g-4}B_1C_5^{-1}B_3^{-1}B_2^{-1}
\]
and \[
F_8 = R^2F_7R^{-2} = C_{g-1}C_{g-2}B_3C_7^{-1}B_5^{-1}B_4^{-1}.
\]

Then, $F_8F_7(c_{g-1},c_{g-2},b_3,c_7,b_5,b_4)=(c_{g-1},c_{g-2},b_3,c_7,c_5,b_4)$ so
that\\
$F_9=C_{g-1}C_{g-2}B_3C_7^{-1}C_5^{-1}B_4^{-1}$$\in$$H$.

Then, $F_9F_8^{-1}= C_5B_5^{-1}$$\in$$H$ and then
\mbox{$C_iB_i^{-1}$$\in$$H$ $\forall i$}.

Let \[
F_{10} = F_1(B_{g-3}C_{g-3}^{-1}) = A_1B_2C_4C_{g-1}^{-1}C_{g-3}^{-1}A_{g-4}^{-1}
\]
and \[
F_{11} = RF_{10}R^{-1} = A_2B_3C_5C_g^{-1}C_{g-2}^{-1}A_{g-3}^{-1}.
\]

Then,
$F_{11}F_{10}(a_2,b_3,c_5,c_g,c_{g-2},a_{g-3})=(b_2,b_3,c_5,c_g,c_{g-2},a_{g-3})$\\
so that $F_{12}=B_2B_3C_5C_g^{-1}C_{g-2}^{-1}A_{g-3}^{-1}$$\in$$H$.

Then, $F_{12}F_{11}^{-1}= B_2A_2^{-1}$$\in$$H$ and then
\mbox{$B_iA_i^{-1}$$\in$$H$ $\forall i$}.

It follows from Corollary~\ref{cor:thmKorkmaz} that
$H=\mod(\Sigma_g)$, completing the proof of the lemma.
\end{proof}

\begin{lemma} \label{lem:gen10} 
If $g \geq 13$, then the mapping class group $\mod(\Sigma_g)$ is generated 
by two elements $R$ and $A_1B_2C_4C_{g-2}^{-1}B_{g-4}^{-1}A_{g-5}^{-1}$.
\end{lemma}

\begin{proof}
Let $F_1=A_1B_2C_4C_{g-2}^{-1}B_{g-4}^{-1}A_{g-5}^{-1}$. Let us denote by
$H$ the subgroup of $\mod(\Sigma_g)$ generated by the set
$\{ R, F_1\}$.

%
Let
\[
F_2 = RF_1R^{-1} = A_2B_3C_5C_{g-1}^{-1}B_{g-3}^{-1}A_{g-4}^{-1}.
\]

Then,
$F_2F_1(a_2,b_3,c_5,c_{g-1},b_{g-3},a_{g-4})=(b_2,b_3,c_5,c_{g-1},b_{g-3},b_{g-4})$
so that $F_3=B_2B_3C_5C_{g-1}^{-1}B_{g-3}^{-1}B_{g-4}^{-1}\in H$.

Let \[
F_4 = F_2F_3^{-1} = A_2B_2^{-1}A_{g-4}^{-1}B_{g-4}
\]
\[
F_5 = RF_4R^{-1} = A_3B_3^{-1}A_{g-3}^{-1}B_{g-3}
\]
\[
F_6 = F_5F_3 = B_2A_3C_5C_{g-1}^{-1}A_{g-3}^{-1}B_{g-4}^{-1}
\]
\[
F_7 = R^{-2}F_6R^2 = B_gA_1C_3C_{g-3}^{-1}A_{g-5}^{-1}B_{g-6}^{-1}
\]
and \[
F_8 = F_7^{-1} = B_{g-6}A_{g-5}C_{g-3}C_3^{-1}A_1^{-1}B_g^{-1}.
\]

Then,
$F_8F_6(b_{g-6},a_{g-5},c_{g-3},c_3,a_1,b_g)=(b_{g-6},a_{g-5},c_{g-3},c_3,a_1,c_{g-1})$
so that $F_9=B_{g-6}A_{g-5}C_{g-3}C_3^{-1}A_1^{-1}C_{g-1}^{-1}$$\in$$H$.

$F_9F_8^{-1}=C_{g-1}B_g^{-1}$$\in$$H$ and then by conjugating $C_{g-1}B_g^{-1}$ with
$R$
iteratively, we get $C_iB_{i+1}^{-1}$$\in$$H$ $\forall i$.

Let \[
F_{10} = F_3(C_{g-1}B_g^{-1}) = B_2B_3C_5B_g^{-1}B_{g-3}^{-1}B_{g-4}^{-1}
\]
and \[
F_{11} = R^2F_{10}R^{-2} = B_4B_5C_7B_2^{-1}B_{g-1}^{-1}B_{g-2}^{-1}.
\]

Then,
$F_{11}F_{10}(b_4,b_5,c_7,b_2,b_{g-1},b_{g-2})=(b_4,c_5,c_7,b_2,b_{g-1},b_{g-2})$\\
so that $F_{12}=B_4C_5C_7B_2^{-1}B_{g-1}^{-1}B_{g-2}^{-1}$$\in$$H$.

Then, $F_{12}F_{11}^{-1}= C_5B_5^{-1}$$\in$$H$ and then
\mbox{$C_iB_i^{-1}$$\in$$H$ $\forall i$}.

Let \[
F_{13} = F_1(B_{g-4}C_{g-4}^{-1}) = A_1B_2C_4C_{g-2}^{-1}C_{g-4}^{-1}A_{g-5}^{-1}
\]
and \[
F_{14} = RF_{13}R^{-1} = A_2B_3C_5C_{g-1}^{-1}C_{g-3}^{-1}A_{g-4}^{-1}.
\]

Then,
$F_{14}F_{13}(a_2,b_3,c_5,c_{g-1},c_{g-3},a_{g-4})=(b_2,b_3,c_5,c_{g-1},c_{g-3},a_{g-4})$
so that $F_{15}=B_2B_3C_5C_{g-1}^{-1}C_{g-3}^{-1}A_{g-4}^{-1}$$\in$$H$.

Then, $F_{15}F_{14}^{-1}= B_2A_2^{-1}$$\in$$H$ and then
\mbox{$B_iA_i^{-1}$$\in$$H$ $\forall i$}.

It follows from Corollary~\ref{cor:thmKorkmaz} that
$H=\mod(\Sigma_g)$, completing the proof of the corollary.
\end{proof}

\begin{lemma} \label{lem:gen11} 
If $k \geq 7$ and $g \geq 2k+1$, then the mapping class group $\mod(\Sigma_g)$ is generated 
by elements $R$ and $A_1B_2C_4C_{g-k+4}^{-1}B_{g-k+2}^{-1}A_{g-k+1}^{-1}$.
\end{lemma}

\begin{proof}
Let $F_1=A_1B_2C_4C_{g-k+4}^{-1}B_{g-k+2}^{-1}A_{g-k+1}^{-1}$. Let us denote by
$H$ the subgroup of $\mod(\Sigma_g)$ generated by the set
$\{ R, F_1\}$.

%
Let
\[
F_2 = R^{k-3}F_1R^{3-k} = A_{k-2}B_{k-1}C_{k+1}C_1^{-1}B_{g-1}^{-1}A_{g-2}^{-1}
\]
and \[
F_3 = F_2^{-1} = A_{g-2}B_{g-1}C_1C_{k+1}^{-1}B_{k-1}^{-1}A_{k-2}^{-1}.
\]

$F_3F_1(a_{g-2},b_{g-1},c_1,c_{k+1},b_{k-1},a_{k-2})=(a_{g-2},b_{g-1},b_2,c_{k+1},b_{k-1},a_{k-2})$ so
that $F_4=A_{g-2}B_{g-1}B_2C_{k+1}^{-1}B_{k-1}^{-1}A_{k-2}^{-1}\in H$.

$F_4F_3^{-1}=B_2C_1^{-1}$$\in$$H$ and then by conjugating $B_2C_1^{-1}$ with $R$
iteratively, we get $B_{i+1}C_i^{-1}$$\in$$H$ $\forall i$.

Let \[
F_5 = F_1(B_{g-k+2}C_{g-k+1}^{-1}) = A_1B_2C_4C_{g-k+4}^{-1}C_{g-k+1}^{-1}A_{g-k+1}^{-1}
\]
and \[
F_6 = RF_5R^{-1} = A_2B_3C_5C_{g-k+5}^{-1}C_{g-k+2}^{-1}A_{g-k+2}^{-1}.
\]

$F_6F_5(a_2,b_3,c_5,c_{g-k+5},c_{g-k+2},a_{g-k+2})=(b_2,b_3,c_5,c_{g-k+5},c_{g-k+2},a_{g-k+2})$\\
so that $F_7=B_2B_3C_5C_{g-k+5}^{-1}C_{g-k+2}^{-1}A_{g-k+2}^{-1}$$\in$$H$.

Then, $F_7F_6^{-1}= B_2A_2^{-1}$$\in$$H$ and then
\mbox{$B_iA_i^{-1}$$\in$$H$ $\forall i$}.

Let \[
F_8 = R^{k-2}F_6R^{2-k} = A_kB_{k+1}C_{k+3}C_3^{-1}Cg^{-1}A_g^{-1}
\]
and \[
F_9 = F_8^{-1} = A_gC_gC_3C_{k-3}^{-1}B_{k+1}^{-1}A_k^{-1}.
\]

Then,
$F_9F_6(a_g,c_g,c_3,c_{k+3},b_{k+1},a_k)=(a_g,c_g,b_3,c_{k+3},b_{k+1},a_k)$\\
so that $F_{10}=A_gC_gB_3C_{k-3}^{-1}B_{k+1}^{-1}A_k^{-1}$$\in$$H$.

Then, $F_{10}F_9^{-1}= B_3C_3^{-1}$$\in$$H$ and then
\mbox{$B_iC_i^{-1}$$\in$$H$ $\forall i$}.

It follows from Corollary~\ref{cor:thmKorkmaz} that
$H=\mod(\Sigma_g)$, completing the proof of the lemma.
\end{proof}

\begin{corollary} \label{cor:gen8} 
If $k \geq 5$ and $g \geq 2k+1$, then the mapping class group $\mod(\Sigma_g)$ is generated 
by elements $R$ and $A_1B_2C_4C_{g-k+4}^{-1}B_{g-k+2}^{-1}A_{g-k+1}^{-1}$.
\end{corollary}

\begin{proof}
It directly follows from Lemma~\ref{lem:gen9}, Lemma~\ref{lem:gen10} and Lemma~\ref{lem:gen11}.
\end{proof}

\bigskip

\section{Main Results}

\begin{lemma} \label{lem:order}
If $R$ is an element of order k in a group $G$ and if $x$ and $y$ are 
elements in $G$ satisfying $RxR^{-1}=y$, then the order of $Rxy^{-1}$ is also k.
 \end{lemma}
\begin{proof}
$(Rxy^{-1})^{k}=(yRy^{-1})^{k}=yR^ky^{-1}=1$.\\
On the other hand, if $(Rxy^{-1})^{l}=1$ then $(Rxy^{-1})^{l}=(yRy^{-1})^{l}=yR^ly^{-1}=1$ i.e. $R^l=1$ and hence $k \mid l$.
\end{proof}

Now, we can prove Theorem~\ref{thm:1}.

\begin{proof}
If $g=10$, let $H_{10}$  be the subgroup of $\mod(\Sigma_{10})$ generated by the set $\{ R, R^4A_1C_1B_3B_7^{-1}C_5^{-1}A_5^{-1}\}$.
Then, $H_{10} =$ $\mod(\Sigma_{10})$ by Corollary ~\ref{cor:gen5}. 
Then, we are done by Lemma ~\ref{lem:order} since $R^4(A_1C_1B_3)R^{-4}=A_5C_5B_7$.
Note that, order of $R^4$ is clearly $5$ and hence order of the element $R^4(A_1C_1B_3)(A_5C_5B_7)^{-1}$ is also $5$ by Lemma ~\ref{lem:order}
since $R^4(a_1)=a_5$, $R^4(c_1)=c_5$ and $R^4(b_3)=b_7$ implies $R^4(A_1C_1B_3)R^{-4} = A_5C_5B_7$.

If $g=9$, let $H_9$  be the subgroup of $\mod(\Sigma_9)$ generated by the set $\{ R, R^3B_1A_3C_5C_8^{-1}A_6^{-1}B_4^{-1}\}$.
Then, $H_9 =$ $\mod(\Sigma_9)$ by Corollary ~\ref{cor:gen13}. 
Then, we are done by Lemma ~\ref{lem:order} since $R^3(B_1A_3C_5)R^{-3}=B_4A_6C_8$.

If $g=8$, let $H_8$  be the subgroup of $\mod(\Sigma_8)$ generated by the set $\{ R, R^2B_1A_5C_5C_7^{-1}A_7^{-1}B_3^{-1}\}$.
Then, $H_8 =$ $\mod(\Sigma_8)$ by Corollary ~\ref{cor:gen12}.
Then, we are done by Lemma ~\ref{lem:order} since $R^2(B_1A_5C_5)R^{-2}=B_3A_7C_7$.

If $g=7$, let $H_7$ be the subgroup of $\mod(\Sigma_7)$ generated by the set $\{ R, RC_1B_4A_6A_7^{-1}B_5^{-1}C_2^{-1}\}$.
Then, $H_7 =$ $\mod(\Sigma_7)$ by Corollary ~\ref{cor:gen2}. 
Then, we are done by Lemma ~\ref{lem:order} since $R(C_1B_4A_6)R^{-1}=C_2B_5A_7$.

The remaining part of the proof is the case of $g\geq 11$.
Let $k=g/g'$ so that $k$ is the greatest divisor of $g$ such that $k$ is strictly less than $g/2$.
Clearly, k can be any positive integer but three.

If $k=2$, let $K_2$ be the subgroup of $\mod(\Sigma_g)$ generated by the set $\{ R, R^2A_1B_4C_8C_{10}^{-1}B_6^{-1}A_3^{-1}\}$.
Then, $K_2 =$ $\mod(\Sigma_g)$ by Corollary ~\ref{cor:gen6}. 
Then, we are done by Lemma ~\ref{lem:order} since $R^2(A_1B_4C_8)R^{-2}=A_3B_6C_{10}$.

If $k=4$, let $K_4$ be the subgroup of $\mod(\Sigma_g)$ generated by the set $\{ R, R^4B_1A_3C_6C_{10}^{-1}A_7^{-1}B_5^{-1}\}$.
Then, $K_4 =$ $\mod(\Sigma_g)$ by Corollary ~\ref{cor:gen7}. 
Then, we are done by Lemma ~\ref{lem:order} since $R^4(B_1A_3C_6)R^{-4}=B_5A_7C_{10}$.

If $k=1$ or $k=5$, let $K_5$ be the subgroup of $\mod(\Sigma_g)$ generated by the set $\{R, R^{-5}A_1B_2C_4C_{g-1}^{-1}B_{g-3}^{-1}A_{g-4}^{-1}\}$.
Then, $K_5 =$ $\mod(\Sigma_g)$ by Corollary ~\ref{cor:gen8}. 
Then, we are done by Lemma ~\ref{lem:order} since $R^{-5}(A_1B_2C_4)R^5=A_{g-4}B_{g-3}C_{g-1}$.

If $k=6$, let $K_6$ be the subgroup of $\mod(\Sigma_g)$ generated by the set $\{ R, R^{-6}A_1B_2C_4C_{g-2}^{-1}B_{g-4}^{-1}A_{g-5}^{-1}\}$.
Then, $K_6 =$ $\mod(\Sigma_g)$ by Corollary ~\ref{cor:gen8}. 
Then, we are done by Lemma~\ref {lem:order} since $R^{-6}(A_1B_2C_4)R^6=A_{g-5}B_{g-4}C_{g-2}$.

If $k\geq 7$, let $K$ be the subgroup of $\mod(\Sigma_g)$ generated by the set $\{ R, R^{-k}A_1B_2C_4C_{g-k+4}^{-1}B_{g-k+2}^{-1}A_{g-k+1}^{-1}\}$.
Then, $K =$ $\mod(\Sigma_g)$ by Corollary ~\ref{cor:gen8}. 
Then, we are done by Lemma~\ref {lem:order} since $R^{-k}(A_1B_2C_4)R^k = A_{g-k+1}B_{g-k+2}C_{g-k+4}$.
\end{proof}

Finally, we prove Theorem~\ref{thm:2}.
\begin{proof}
If $g=6$, let $H_6$ be the subgroup of $\mod(\Sigma_6)$ generated by the set $\{ R, RC_1B_4A_6A_1^{-1}B_5^{-1}C_2^{-1}\}$.
Then, $H_6 =$ $\mod(\Sigma_6)$ by Corollary ~\ref{cor:gen1}. 
Then, we are done by Lemma ~\ref{lem:order} since $R(C_1B_4A_6)R^{-1}=C_2B_5A_1$.
Note that, since $R(c_1)=c_2$, $R(b_4)=b_5$ and $R(a_6)=a_1$, we have
$R(C_1B_4A_6)R^{-1} = C_2B_5A_1$
which implies order of the element $R(C_1B_4A_6)(C_2B_5A_1)^{-1}$ is $g$.

If $g=7$, let $H_7$ be the subgroup of $\mod(\Sigma_7)$ generated by the set $\{ R, RC_1B_4A_6A_7^{-1}B_5^{-1}C_2^{-1}\}$.
Then, $H_7 =$ $\mod(\Sigma_7)$ by Corollary ~\ref{cor:gen2}. 
Then, we are done by Lemma ~\ref{lem:order} since $R(C_1B_4A_6)R^{-1}=C_2B_5A_7$.

If $g=8$, let $H_8$ be the subgroup of $\mod(\Sigma_8)$ generated by the set $\{ R, RB_1C_4A_7A_8^{-1}C_5^{-1}B_2^{-1}\}$.
Then, $H_8 =$ $\mod(\Sigma_8)$ by Corollary ~\ref{cor:gen3}. 
Then, we are done by Lemma ~\ref{lem:order} since $R(B_1C_4A_7)R^{-1}=B_2C_5A_8$.

If $g\geq 9$, let $H_9$ be the subgroup of $\mod(\Sigma_g)$ generated by the set $\{ R, RC_1B_4A_7A_8^{-1}B_5^{-1}C_2^{-1}\}$.
Then, $H_9 =$ $\mod(\Sigma_g)$ by Corollary ~\ref{cor:gen4}. 
Then, we are done by Lemma ~\ref{lem:order} since $R(C_1B_4A_7)R^{-1}=C_2B_5A_8$.
%
%
%
\end{proof}

\section{Further Results}

In this section, we prove Theorem ~\ref{thm:3} which states as: for $g\geq 3k^2+4k+1$ and any positive integer $k$, the mapping class group $\mod(\Sigma_g)$ is generated by 
two elements of order $g/\gcd (g,k)$. 

Korkmaz showed the following in the proof of Theorem ~\ref{thm:thmKorkmaz}.
\begin{theorem} \label{thm:thmKorkmaz2}
If $g\geq 3$, then the mapping class group $\mod(\Sigma_g)$ is generated by the elements
\(
A_iA_j^{-1},B_iB_j^{-1}, C_iC_j^{-1} \forall i, j. \)
 \end{theorem}
Sketch of the proof is as follows:
$A_1A_2^{-1}B_1B_2^{-1}(a_1,a_3) = (b_1,a_3)$.
$B_1A_3^{-1}C_1C_2^{-1}(b_1,a_3) = (c_1,a_3)$.
Then, Korkmaz showed that $A_3$ can be generated by these elements by using lantern relation.
Hence, $A_i = (A_iA_3^{-1})A_3$,
$B_i = (B_iB_1^{-1})(B_1A_3^{-1})A_3$ and
$C_i = (C_iC_1^{-1})(C_1A_3^{-1})A_3$ are generated by given elements.
This finishes the proof.

Now, we prove the next statement as a corollary to Theorem ~\ref{thm:thmKorkmaz2}.
\begin{corollary} \label{generating}
If $g\geq 3$, then the mapping class group $\mod(\Sigma_g)$ is generated by the elements
\(
A_iB_i^{-1},C_iB_i^{-1}, C_iB_{i+1}^{-1} \forall i. \)
 \end{corollary}
\begin{proof}
Let us denote by $H$ the subgroup generated by the elements 
\(
A_iB_i^{-1},C_iB_i^{-1}, C_iB_{i+1}^{-1} \forall i. \)
\\
$B_iB_j^{-1}=(B_iC_i^{-1})(C_iB_{i+1}^{-1})\cdots (B_{j-1}C_{j-1}^{-1})(C_{j-1}B_j^{-1})$$\in$$H$$\forall i, j$\\
$C_iC_j^{-1}=(C_iB_i^{-1})(B_iB_j^{-1})(B_jC_j^{-1})$$\in$$H$ $ \forall i, j$\\
$A_iA_j^{-1}=(A_iB_i^{-1})(B_iB_j^{-1})(B_jA_j^{-1})$$\in$$H$ $ \forall i, j$

It follows from  Theorem~\ref{thm:thmKorkmaz2} that
$H=\mod(\Sigma_g)$, completing the proof of the lemma.
\end{proof}

\begin{theorem} \label{thm:r2}
If $g\geq 21$, then the mapping class group $\mod(\Sigma_g)$ is generated by the elements
\(
R^2,B_1B_2A_5A_8C_{11}C_{14}C_{16}^{-1}C_{13}^{-1}A_{10}^{-1}A_7^{-1}B_4^{-1}B_3^{-1}. \)
 \end{theorem}

\begin{proof}

Let $F_1=B_1B_2A_5A_8C_{11}C_{14}C_{16}^{-1}C_{13}^{-1}A_{10}^{-1}A_7^{-1}B_4^{-1}B_3^{-1}$. Let us denote by
$H$ the subgroup of $\mod(\Sigma_g)$ generated by the set
$\{ R^2, F_1\}$.

Let
\[
F_2 = R^2F_1R^{-2} = B_3B_4A_7A_{10}C_{13}C_{16}C_{18}^{-1}C_{15}^{-1}A_{12}^{-1}A_9^{-1}B_6^{-1}B_5^{-1}.
\]
and \[
F_3 = F_2^{-1} = B_5B_6A_9A_{12}C_{15}C_{18}C_{16}^{-1}C_{13}^{-1}A_{10}^{-1}A_7^{-1}B_4^{-1}B_3^{-1}.
\]

Then, $F_3F_1(b_5,b_6, \cdots ,b_3)=(a_5,b_6, \cdots ,b_3)$ so that\\
$F_4=A_5B_6A_9A_{12}C_{15}C_{18}C_{16}^{-1}C_{13}^{-1}A_{10}^{-1}A_7^{-1}B_4^{-1}B_3^{-1}\in H$.

Note that $\cdots$ refers to the elements remaining fixed under the given maps.

$F_4F_3^{-1}=A_5B_5^{-1}$$\in$$H$ and then by conjugating $A_5B_5^{-1}$ with $R^2$ iteratively, we get $A_{2i+1}B_{2i+1}^{-1}$$\in$$H$ $\forall i$.

Let
\[
F_5 = R^4F_1R^{-4} = B_5B_6A_9A_{12}C_{15}C_{18}C_{20}^{-1}C_{17}^{-1}A_{14}^{-1}A_{11}^{-1}B_8^{-1}B_7^{-1}.
\]
and \begin{eqnarray*}
F_6
&=& (A_7B_7^{-1})F_5^{-1}(B_5A_5^{-1})\\
&=& A_7B_8A_{11}A_{14}C_{17}C_{20}C_{18}^{-1}C_{15}^{-1}A_{12}^{-1}A_9^{-1}B_6^{-1}A_5^{-1}.
\end{eqnarray*}

Then, $F_6F_1(a_7,b_8,a_{11}, \cdots ,b_6,a_5)=(a_7,a_8,a_{11}, \cdots ,b_6,a_5)$ so that\\
$F_7=A_7A_8A_{11}A_{14}C_{17}C_{20}C_{18}^{-1}C_{15}^{-1}A_{12}^{-1}A_9^{-1}B_6^{-1}A_5^{-1}\in H$.

$F_7F_6^{-1}=A_8B_8^{-1}$$\in$$H$ and then by conjugating $A_8B_8^{-1}$ with $R^2$ iteratively, we get $A_{2i}B_{2i}^{-1}$$\in$$H$ $\forall i$.\\
Hence, $A_iB_i^{-1}$$\in$$H$ $\forall i$.

Let
\[
F_8 = (B_{12}A_{12}^{-1})F_4 = A_5B_6A_9B_{12}C_{15}C_{18}C_{16}^{-1}C_{13}^{-1}A_{10}^{-1}A_7^{-1}B_4^{-1}B_3^{-1}.
\]

Then, $F_8F_1( \cdots ,b_{12}, \cdots )=( \cdots ,c_{11}, \cdots )$ so that\\
$F_9=A_5B_6A_9C_{11}C_{15}C_{18}C_{16}^{-1}C_{13}^{-1}A_{10}^{-1}A_7^{-1}B_4^{-1}B_3^{-1}\in H$.

$F_9F_8^{-1}=C_{11}B_{12}^{-1}$$\in$$H$ and then by conjugating $C_{11}B_{12}^{-1}$ with $R^2$ iteratively, we get $C_{2i+1}B_{2i+2}^{-1}$$\in$$H$ $\forall i$.

Let \[
F_{10} = (B_{11}A_{11}^{-1})F_7 = A_7A_8B_{11}A_{14}C_{17}C_{20}C_{18}^{-1}C_{15}^{-1}A_{12}^{-1}A_9^{-1}B_6^{-1}A_5^{-1}. 
\]

Then, $F_{10}F_1( \cdots ,b_{11}, \cdots )=( \cdots ,c_{11}, \cdots )$ so that\\
$F_{11}=A_7A_8C_{11}A_{14}C_{17}C_{20}C_{18}^{-1}C_{15}^{-1}A_{12}^{-1}A_9^{-1}B_6^{-1}A_5^{-1}$$\in$$H$.

Then, $F_{11}F_{10}^{-1}= C_{11}B_{11}^{-1}$$\in$$H$ and then, we get $C_{2i+1}B_{2i+1}^{-1}$$\in$$H$ $\forall i$.

Let
\[
F_{12} = (B_{15}C_{15}^{-1})F_4 = A_5B_6A_9A_{12}B_{15}C_{18}C_{16}^{-1}C_{13}^{-1}A_{10}^{-1}A_7^{-1}B_4^{-1}B_3^{-1}.
\]

Then, $F_{12}F_1( \cdots ,b_{15}, \cdots )=( \cdots ,c_{14}, \cdots )$ so that\\
$F_{13}=A_5B_6A_9A_{12}C_{14}C_{18}C_{16}^{-1}C_{13}^{-1}A_{10}^{-1}A_7^{-1}B_4^{-1}B_3^{-1}$$\in$$H$.

Then, $F_{13}F_{12}^{-1}= C_{14}B_{15}^{-1}$$\in$$H$ and then, we get $C_{2i}B_{2i+1}^{-1}$$\in$$H$ $\forall i$.\\
Hence, $C_iB_{i+1}^{-1}$$\in$$H$ $\forall i$.

Let
\[
F_{14} = F_7(C_{15}B_{16}^{-1}) = A_7A_8A_{11}A_{14}C_{17}C_{20}C_{18}^{-1}B_{16}^{-1}A_{12}^{-1}A_9^{-1}B_6^{-1}A_5^{-1}.
\]

Then, $F_{14}F_1(\cdots ,b_{16}, \cdots )=( \cdots ,c_{16}, \cdots )$ so that\\
$F_{15}=A_7A_8A_{11}A_{14}C_{17}C_{20}C_{18}^{-1}C_{16}^{-1}A_{12}^{-1}A_9^{-1}B_6^{-1}A_5^{-1}$$\in$$H$.

Then, $F_{15}^{-1}F_{14}= C_{16}B_{16}^{-1}$$\in$$H$ and then, we get $C_{2i}B_{2i}^{-1}$$\in$$H$ $\forall i$.\\
Hence, $C_iB_i^{-1}$$\in$$H$ $\forall i$.

It follows from  Corollary~\ref{generating} that
$H=\mod(\Sigma_g)$, completing the proof of the theorem.
\end{proof}

\begin{corollary} \label{cor:r2}
If $g$ is even and $g\geq 22$, then the mapping class group $\mod(\Sigma_g)$ is generated by two elements
of order $g/2$.
\end{corollary}

\begin{proof}
Let $H$ be the subgroup of $\mod(\Sigma_g)$ generated by the set\\ 
$\{ 
R^2, R^2B_1B_2A_5A_8C_{11}C_{14}C_{16}^{-1}C_{13}^{-1}A_{10}^{-1}A_7^{-1}B_4^{-1}B_3^{-1}\}$. \\
Then, $H=$ $\mod(\Sigma_g)$ by Theorem ~\ref{thm:r2}. 
Then, we are done by Lemma ~\ref{lem:order} since
$R^2(B_1B_2A_5A_8C_{11}C_{14})R^{-2}=B_3B_4A_7A_{10}C_{13}C_{16}$.
\end{proof}

Generalization of Theorem ~\ref{thm:r2} and Corollary ~\ref{cor:r2} is as follows:
\begin{theorem} \label{generalization}
For $k\geq 2$ and $g\geq 3k^2+4k+1$, the mapping class group $\mod(\Sigma_g)$ is generated by the elements
\(
R^k,R^kF(R^kF^{-1}R^{-k})\) where $F=B_1B_2\cdots B_kA_{2k+1}A_{3k+2}\cdots A_{k^2+2k}C_{k^2+3k+1}C_{k^2+4k+2}\cdots C_{2k^2+3k}$.
 \end{theorem}
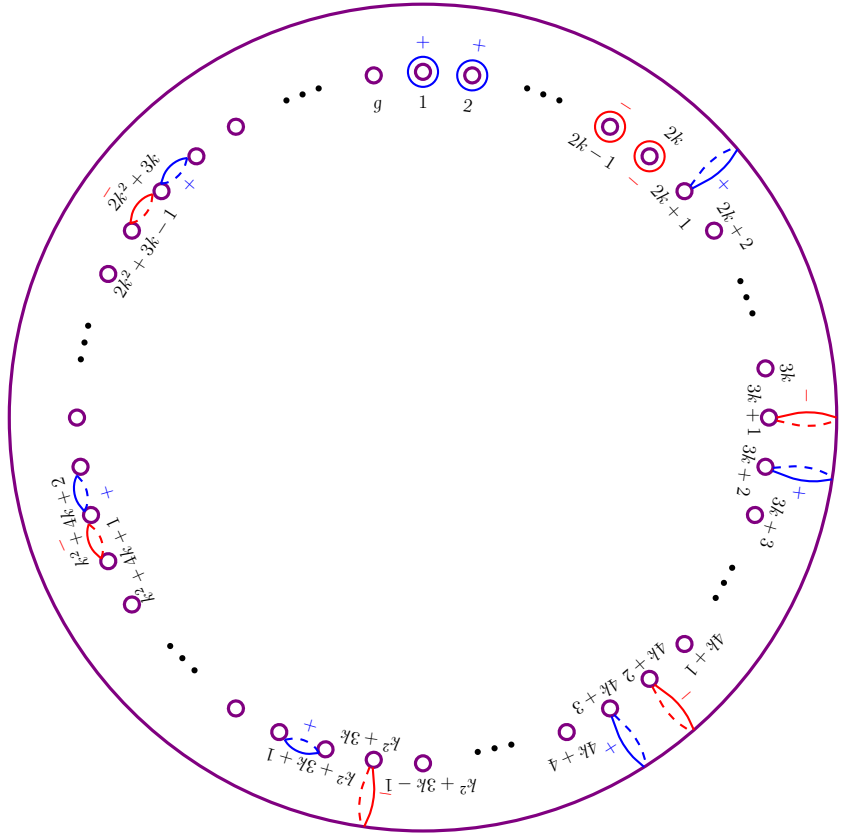
\begin{figure}
\begin{tikzpicture}[scale=1.0]
\begin{scope} [xshift=0cm, yshift=0cm]
 \draw[very thick, violet] (0,0) circle [radius=5.5cm];
 \draw[very thick, violet] (0,4.6) circle [radius=0.1cm]; 
 \draw[very thick, violet, rotate=-360/44] (0,4.6) circle [radius=0.1cm]; 
 
 \draw[very thick, violet, rotate=-360/44*4] (0,4.6) circle [radius=0.1cm];  
 \draw[very thick, violet, rotate=-360/44*5] (0,4.6) circle [radius=0.1cm];  
 \draw[very thick, violet, rotate=-360/44*6] (0,4.6) circle [radius=0.1cm];  
 \draw[very thick, violet, rotate=-360/44*7] (0,4.6) circle [radius=0.1cm]; 
 
 \draw[very thick, violet, rotate=-360/44*10] (0,4.6) circle [radius=0.1cm];  
 \draw[very thick, violet, rotate=-360/44*11] (0,4.6) circle [radius=0.1cm];  
 \draw[very thick, violet, rotate=-360/44*12] (0,4.6) circle [radius=0.1cm];  
 \draw[very thick, violet, rotate=-360/44*13] (0,4.6) circle [radius=0.1cm]; 
 
 \draw[very thick, violet, rotate=-360/44*16] (0,4.6) circle [radius=0.1cm];  
 \draw[very thick, violet, rotate=-360/44*17] (0,4.6) circle [radius=0.1cm];  
 \draw[very thick, violet, rotate=-360/44*18] (0,4.6) circle [radius=0.1cm];  
 \draw[very thick, violet, rotate=-360/44*19] (0,4.6) circle [radius=0.1cm]; 
 
 \draw[very thick, violet, rotate=-360/44*22] (0,4.6) circle [radius=0.1cm]; 
 \draw[very thick, violet, rotate=-360/44*23] (0,4.6) circle [radius=0.1cm];  
 \draw[very thick, violet, rotate=-360/44*24] (0,4.6) circle [radius=0.1cm];  
 \draw[very thick, violet, rotate=-360/44*25] (0,4.6) circle [radius=0.1cm];  
 \draw[very thick, violet, rotate=-360/44*26] (0,4.6) circle [radius=0.1cm]; 
  
 \draw[very thick, violet, rotate=-360/44*29] (0,4.6) circle [radius=0.1cm]; 
 \draw[very thick, violet, rotate=-360/44*30] (0,4.6) circle [radius=0.1cm];  
 \draw[very thick, violet, rotate=-360/44*31] (0,4.6) circle [radius=0.1cm];  
 \draw[very thick, violet, rotate=-360/44*32] (0,4.6) circle [radius=0.1cm];  
 \draw[very thick, violet, rotate=-360/44*33] (0,4.6) circle [radius=0.1cm]; 
 
 \draw[very thick, violet, rotate=-360/44*36] (0,4.6) circle [radius=0.1cm]; 
 \draw[very thick, violet, rotate=-360/44*37] (0,4.6) circle [radius=0.1cm];  
 \draw[very thick, violet, rotate=-360/44*38] (0,4.6) circle [radius=0.1cm];  
 \draw[very thick, violet, rotate=-360/44*39] (0,4.6) circle [radius=0.1cm];  
 \draw[very thick, violet, rotate=-360/44*40] (0,4.6) circle [radius=0.1cm]; 
 
 \draw[very thick, violet, rotate=360/44] (0,4.6) circle [radius=0.1cm]; 
 \draw[thick, blue, rotate=-360/44*26, rounded corners=6pt] (-1.25,1.4+3.1)--(-1.0, 1.6+3.1)--(-0.77,1.5+3.1);  
 \draw[thick, blue, dashed,  rotate=-360/44*26, rounded corners=6pt] (-1.25,1.4+3.1)--(-1.0, 1.3+3.1)--(-0.77,1.5+3.1);

  \draw[thick, blue, rotate=-360/44*33, rounded corners=6pt]  (-1.25,1.4+3.1)--(-1.0, 1.6+3.1)--(-0.77,1.5+3.1);  
 \draw[thick, blue, dashed,  rotate=-360/44*33, rounded corners=6pt]  (-1.25,1.4+3.1)--(-1.0, 1.3+3.1)--(-0.77,1.5+3.1); 
  \draw[thick, red, rotate=-360/44*32, rounded corners=6pt]  (-1.25,1.4+3.1)--(-1.0, 1.6+3.1)--(-0.77,1.5+3.1);  
 \draw[thick, red, dashed,  rotate=-360/44*32, rounded corners=6pt]  (-1.25,1.4+3.1)--(-1.0, 1.3+3.1)--(-0.77,1.5+3.1); 
 
   \draw[thick, blue, rotate=-360/44*40, rounded corners=6pt]  (-1.25,1.4+3.1)--(-1.0, 1.6+3.1)--(-0.77,1.5+3.1); 
 \draw[thick, blue, dashed,  rotate=-360/44*40, rounded corners=6pt]  (-1.25,1.4+3.1)--(-1.0, 1.3+3.1)--(-0.77,1.5+3.1); 
  \draw[thick, red, rotate=-360/44*39, rounded corners=6pt]  (-1.25,1.4+3.1)--(-1.0, 1.6+3.1)--(-0.77,1.5+3.1);  
 \draw[thick, red, dashed,  rotate=-360/44*39, rounded corners=6pt]  (-1.25,1.4+3.1)--(-1.0, 1.3+3.1)--(-0.77,1.5+3.1);   
 
 \draw[thick, blue, rotate=-360/44] (0,4.6) circle [radius=0.2cm]; 
\draw[thick, red, rotate=-360/44*5] (0,4.6) circle [radius=0.2cm]; 
 \draw[thick, blue, rotate=0] (0,4.6) circle [radius=0.2cm]; 
\draw[thick, red, rotate=-360/44*4] (0,4.6) circle [radius=0.2cm]; 
 \draw[thick, blue, rotate=-360/44*6, rounded corners=6pt] (0.04,4.7)--(0.14, 5.15)--(0.04,5.5);  
 \draw[thick, blue, dashed, rotate=-360/44*6, rounded corners=6pt]  (0.0,4.7)--(-0.1, 5.15)--(0.0,5.5); 
 
 \draw[thick, dashed, red, rotate=-360/44*11, rounded corners=6pt] (0.04,4.7)--(0.14, 5.15)--(0.04,5.5);  
 \draw[thick, red, rotate=-360/44*11, rounded corners=6pt] (0.0,4.7)--(-0.1, 5.15)--(0.0,5.5);
  \draw[thick, blue, rotate=-360/44*12, rounded corners=6pt] (0.04,4.7)--(0.14, 5.15)--(0.04,5.5);  
 \draw[thick, blue, dashed, rotate=-360/44*12, rounded corners=6pt]  (0.0,4.7)--(-0.1, 5.15)--(0.0,5.5);   
 
  \draw[thick, dashed, red, rotate=-360/44*17, rounded corners=6pt] (0.04,4.7)--(0.14, 5.15)--(0.04,5.5);  
 \draw[thick, red, rotate=-360/44*17, rounded corners=6pt] (0.0,4.7)--(-0.1, 5.15)--(0.0,5.5);
  \draw[thick, blue, rotate=-360/44*18, rounded corners=6pt] (0.04,4.7)--(0.14, 5.15)--(0.04,5.5);  
 \draw[thick, blue, dashed, rotate=-360/44*18, rounded corners=6pt]  (0.0,4.7)--(-0.1, 5.15)--(0.0,5.5);
 
 \draw[thick, dashed, red, rotate=-360/44*23, rounded corners=6pt] (0.04,4.7)--(0.14, 5.15)--(0.04,5.5);  
 \draw[thick, red, rotate=-360/44*23, rounded corners=6pt] (0.0,4.7)--(-0.1, 5.15)--(0.0,5.5);
 
  \node[scale=0.6, blue] at (0,5.0) {$+$};
  \node[scale=0.6, blue, rotate=-360/44] at (0.75,4.95) {$+$};
  \node[scale=0.6, red, rotate=-360/44*4] at (2.7,4.15) {$-$};
  \node[scale=0.6, red, rotate=-360/44*5] at (2.8,3.15) {$-$};
  \node[scale=0.6, blue, rotate=-360/44*6] at (4.0,3.1) {$+$};
  \node[scale=0.6, red, rotate=-360/44*11] at (5.1,0.3) {$-$};
  \node[scale=0.6, blue, rotate=-360/44*12] at (5.0,-1.0) {$+$};
  \node[scale=0.6, red, rotate=-360/44*17] at (3.5,-3.7) {$-$};
  \node[scale=0.6, blue, rotate=-360/44*18] at (2.5,-4.4) {$+$};
  \node[scale=0.6, red, rotate=-360/44*23] at (-0.5,-5.0) {$-$};
  \node[scale=0.6, blue, rotate=-360/44*24] at (-1.5,-4.1) {$+$};
  \node[scale=0.6, red, rotate=-360/44*30] at (-4.8,-1.7) {$-$};
  \node[scale=0.6, blue, rotate=-360/44*31] at (-4.2,-1.0) {$+$};
  \node[scale=0.6, red, rotate=-360/44*37] at (-4.2,3.0) {$-$};
  \node[scale=0.6, blue, rotate=-360/44*38] at (-3.1,3.1) {$+$};
  
\end{scope}

  	\node[scale=2, rotate=-360/44*2.5] at (1.6,4.3) {...};
  	\node[scale=2, rotate=-360/44*8.5] at (4.3,1.6) {...};
  	\node[scale=2, rotate=-360/44*14.5] at (4.0,-2.2) {...};
  	  	\node[scale=2, rotate=-360/44*20.5] at (0.95,-4.4) {...};  	
  	  	\node[scale=2, rotate=-360/44*27.5] at (-3.2,-3.2) {...};  	
  	  	\node[scale=2, rotate=-360/44*34.5] at (-4.5,1.0) {...};
  	  	\node[scale=2, rotate=-360/44*41.5] at (-1.6,4.3) {...};
  	
	 \node[scale=0.6, rotate=360/44] at (-0.6,4.15) {$g$};	
	 \node[scale=0.6] at (0.0,4.2) {$1$};
	 \node[scale=0.6, rotate=-360/44] at (0.6,4.15) {$2$};
	 
	 \node[scale=0.6, rotate=-360/44*4] at (2.25,3.55) {$2k-1$};
	 \node[scale=0.6, rotate=-360/44*5] at (3.35,3.75) {$2k$};
	 \node[scale=0.6, rotate=-360/44*6] at (3.3,2.8) {$2k+1$};
	 \node[scale=0.6, rotate=-360/44*7] at (4.15,2.55) {$2k+2$};
	 
	 \node[scale=0.6, rotate=-360/44*10] at (4.85,0.6) {$3k$};
	 \node[scale=0.6, rotate=-360/44*11] at (4.4,0.1) {$3k+1$};
	 \node[scale=0.6, rotate=-360/44*12] at (4.3,-0.7) {$3k+2$};
	 \node[scale=0.6, rotate=-360/44*13] at (4.65,-1.4) {$3k+3$};
	 
	 \node[scale=0.6, rotate=-360/44*16] at (3.75,-3.15) {$4k+1$};
	 \node[scale=0.6, rotate=-360/44*17] at (2.95,-3.25) {$4k+2$};
	 \node[scale=0.6, rotate=-360/44*18] at (2.3,-3.7) {$4k+3$};
	 \node[scale=0.6, rotate=-360/44*19] at (2.0,-4.5) {$4k+4$};
	 
	 \node[scale=0.6, rotate=-360/44*22] at (0.1,-4.9) {$k^2+3k-1$};
	 \node[scale=0.6, rotate=-360/44*23] at (-0.7,-4.3) {$k^2+3k$};
	 \node[scale=0.6, rotate=-360/44*24] at (-1.5,-4.65) {$k^2+3k+1$};
	 
	 \node[scale=0.6, rotate=-360/44*30] at (-3.9,-1.9) {$k^2+4k+1$};
	 \node[scale=0.6, rotate=-360/44*31] at (-4.7,-1.4) {$k^2+4k+2$};
	 
	 \node[scale=0.6, rotate=-360/44*37] at (-3.7,2.2) {$2k^2+3k-1$};
	 \node[scale=0.6, rotate=-360/44*38] at (-3.85,3.15) {$2k^2+3k$};
	 
\end{tikzpicture}
 	\caption{Generator for Theorem~\ref{thm:3}.}
\end{figure}

\begin{proof}
We define an algorithm to prove the desired result.\\
Let $F=B_1B_2\cdots B_iA_{2k+1}A_{3k+2}\cdots A_{k^2+2k}C_{k^2+3k+1}C_{k^2+4k+2}\cdots C_{2k^2+3k}$
and $F_1=F(R^kF^{-1}R^{-k})$. Let us denote by
$H$ the subgroup of $\mod(\Sigma_g)$ generated by the set
$\{ R^k, F_1\}$.

A) Use conjugation of $F_1$ with $R^k, R^{2k} , \cdots , R^{k^2}$ with proper multiplications to get $A_{k+1}B_{k+1}^{-1}$$\in$$H$, $A_{k+2}B_{k+2}^{-1}$$\in$$H$, $\cdots$ , $A_{2k-1}B_{2k-1}^{-1}$$\in$$H$, $A_{2k}B_{2k}^{-1}$$\in$$H$, respectively. Hence, $A_iB_i^{-1}$$\in$$H$ $\forall i$.

B) Follow the next $k$ steps.

1) Use conjugation of $F_1$ with $R^{kl}$ for some positive integers $l$'s with proper multiplications to get $C_{ik+1}B_{ik+1}^{-1}$$\in$$H$ and $C_{ik+1}B_{ik+2}^{-1}$$\in$$H$ $\forall i$.

2) Use conjugation of $F_1$ with $R^{kl}$ for some positive integers $l$'s with proper multiplications to get $C_{ik+2}B_{ik+2}^{-1}$$\in$$H$ and $C_{ik+2}B_{ik+3}^{-1}$$\in$$H$ $\forall i$.\\
$\cdots $

k) Use conjugation of $F_1$ with $R^{kl}$ for some positive integers $l$'s with proper multiplications to get $C_{ik}B_{ik}^{-1}$$\in$$H$ and $C_{ik}B_{ik+1}^{-1}$$\in$$H$ $\forall i$.\\
Hence, $C_iB_i^{-1}$$\in$$H$ and $C_iB_{i+1}^{-1}$$\in$$H$ $\forall i$.

It follows from  Corollary~\ref{generating} that
$H=\mod(\Sigma_g)$, completing the proof of the theorem.

See Theorem ~\ref{thm:r2} for an example usage of the algorithm.
\end{proof}

Now, we prove Theorem~\ref{thm:3}.
\begin{proof}
For $k\geq 2$ and $g\geq 3k^2+4k+1$, let $H$ be the subgroup of $\mod(\Sigma_g)$ generated by the set 
$\{R^k, R^kF(R^kF^{-1}R^{-k})\}$.
Then, $H =$ $\mod(\Sigma_g)$ by Theorem ~\ref{generalization}. 
Hence, we are done by Lemma ~\ref{lem:order} since the orders of $R^k$ and $R^kF(R^kF^{-1}R^{-k})$ are $g/d$ where $d$ is the greatest common divisor of $g$ and $k$.
If $k=1$, we are done by Theorem ~\ref{thm:2}.
\end{proof}


%
%
%
%



\begin{thebibliography}{xxxx}

 \bibitem{Baykur} I. Baykur, M. Korkmaz,
\emph{Mapping class group is generated by two commutators}.
arXiv:1908.11306v1 [math.GT] 29Aug2019.

\bibitem{BrendleFarb} T.E. Brendle, B. Farb,
\emph{Every mapping class group is generated by 6 involutions}.
J. of Algebra \textbf{278} (2004), 187--198.

%
\bibitem{FarbMargalit} B. Farb, D. Margalit,
\emph{A primer on mapping class groups}.
Princeton University Press, 2011.

\bibitem{Dehn} M. Dehn,
\emph{The group of mapping classes}.  In: Papers on Group Theory and Topology.
Springer-Verlag, 1987. Translated from the German by J. Stillwell 
(Die Gruppe der Abbildungsklassen, Acta Math.  \textbf{69} (1938), 135--206).




   
 \bibitem{Du2} X. Du,
 \emph{Generating the extended mapping class group by torsions}.
J. Knot Theory Ramifications \textbf{26} (2017), 1750037 8pp.

\bibitem{Humphries} S. Humphries,
\emph{Generators for the mapping class group}. 
In: Topology of Low-Dimensional Manifolds, Proc. Second Sussex Conf., Chelwood Gate, 1977, 
Lecture Notes in Math., vol. \textbf{722}, Springer-Verlag, 1979, 44--47.

%
\bibitem{Lanier} J. Lanier,
\emph{Generating mapping class groups with elements of fixed finite order}.
 J. Algebra \textbf{511} (2018), 455--470.


\bibitem{Lickorish} W.B.R. Lickorish,
\emph{A finite set of generators for the homeotopy group of a $2$--manifold}.
 Proc. Cambridge Philos. Soc. \textbf{60} (1964), 769--778.

\bibitem{Kassabov} M. Kassabov,
\emph{Generating mapping class groups by involutions}.
 Arxiv math.GT/0311455, v1 25Nov2003.
 
 
 \bibitem{Korkmaz} M. Korkmaz,
\emph{Generating the surface mapping class group by two elements}.
 Trans. Amer. Math. Soc \textbf{357} (2005), 3299--3310.
 
 \bibitem{Korkmaz3inv} M. Korkmaz,
\emph{Mapping class group is generated by three involutions}.
arXiv:1904.08156v2 [math.GT] 14May2019.
 
 \bibitem{Korkmaz2012} M. Korkmaz, 
\emph{Minimal generating sets for the mapping class group of a surface}. 
 Handbook of Teichm\"uller spaces, Volume III, (2012), 441--463.
 
 \bibitem{Lu} N. Lu,
\emph{On the mapping class groups of the closed orientable surfaces}.
Topology Proc. \textbf{13} (1988), 293--324.
 
\bibitem{Luo} F. Luo,
\emph{Torsion elements in the mapping class group of a surface}.
 Arxiv math.GT/0004048, v1 8Apr2000.

\bibitem{Maclachlan} C. Maclachlan,
\emph{Modulus space is simply-connected}.
Proc. Amer. Math. Soc.  \textbf{29} (1971), 85--86.

\bibitem{Margalit} D. Margalit, 
\emph{Problems, questions, and conjectures about mapping class groups}.
Proceedings of Symposia in Pure Mathematics \textbf{102} (2019) 

\bibitem{Monden} N. Monden,
\emph{Generating the mapping class group by torsion elements of small order}.
Math. Proc. Cambridge Philos. Soc.
  \textbf{154} (2013), 41--62.

\bibitem{McCarthyPapadopoulos} J.D. McCarthy, A. Papadopoulos,
\emph{Involutions in surface mapping class groups}.
Enseign. Math.  (2) \textbf{33} (1987), 275--290.

%

\bibitem{Stukow2} M. Stukow,
\emph{Small torsion generating sets for hyperelliptic mapping class groups}.
  Topology and its Applications \textbf{145} (2004), 83--90.


\bibitem{Wajnryb1983} B. Wajnryb,
\emph{A simple presentation for the mapping class group of an orientable surface}.
Israel J. Math.  \textbf{45} (1983), 157--174.

\bibitem{Wajnryb1996} B. Wajnryb,
\emph{Mapping class group of a surface is generated by two elements}.
Topology  \textbf{35} (1996), 377--383.

\bibitem{Yildiz} O. Yildiz,
\emph{Generating the mapping class group by three involutions}.
arXiv:2002.09151v1 [math.GT] 21Feb2020.

\end{thebibliography}
\end{document}